 \documentclass[11pt,leqno]{amsart}
\usepackage{amsmath}
\usepackage{amssymb}
\usepackage{amsthm}
\usepackage{graphicx,color,subfigure}
\usepackage{enumerate}
\usepackage{multirow}
\usepackage{changepage} 
\usepackage{scrextend}

\newtheorem{theorem}{Theorem}[section]
\newtheorem{lemma}{Lemma}[section]
\newtheorem{corollary}{Corollary}[section]

\newtheorem{proposition}{Proposition}[section] 
\newtheorem{definition}{Definition}[section]
\newtheorem{remark}{Remark}[section]

\newif \ifdetails
\detailsfalse

\def\ba{\begin{array}}
\def\ea{\end{array}}
\def\beq{\begin{equation}}
\def\eeq{\end{equation}}
\def\bea{\begin{eqnarray}}
\def\eea{\end{eqnarray}}
\def\beann{\begin{eqnarray*}}
\def\eeann{\end{eqnarray*}}

\def\R{\mathbb{R}}

\def\cH{\mathcal{H}}

\def\cl{\textup{cl}}

\def\Diag{\textup{Diag}}

\def\inte{\textup{int}}

\def\rank{\textup{rank}}

\def\trace{\textup{Tr}}

\def\argmin{\textup{argmin}}

\def\Newton{\textup{Newton}}
\def\rec{\textup{rec}}

\def\exp{\textup{exp}}
\def\img{\textup{img}}

\title{\bf Primal-Dual Interior-Point Methods for Domain-Driven Formulations*}
\author{
Mehdi Karimi \and Levent Tun\c{c}el}
\thanks{* Some of the material in this manuscript appeared in a preliminary form in Karimi's PhD thesis \cite{karimi_thesis}. \\
Mehdi Karimi (m7karimi@uwaterloo.ca) and Levent Tun\c{c}el (ltuncel@math.uwaterloo.ca): Department of
Combinatorics and Optimization, University
of Waterloo, Waterloo, Ontario N2L 3G1, Canada. Research of the authors was supported in part by Discovery Grants from NSERC and by U.S. Office of Naval Research under award numbers: N00014-12-1-0049, N00014-15-1-2171 and N00014-18-1-2078.}

\begin{document}
\begin{abstract}
We study infeasible-start primal-dual interior-point methods for convex optimization problems given in a typically natural form we denote as Domain-Driven formulation.  Our algorithms extend many advantages of primal-dual interior-point techniques available for conic formulations, such as the current best complexity bounds, and more robust certificates of approximate optimality, unboundedness, and infeasibility, to Domain-Driven formulations. The complexity results are new for the infeasible-start setup used, even in the case of linear programming. In addition to complexity results, our algorithms aim for expanding the applications of, and software for interior-point methods to wider classes of problems beyond optimization over symmetric cones.
\end{abstract}
\maketitle

\pagestyle{myheadings} \thispagestyle{plain}
\markboth{KARIMI and TUN{\c C}EL}
{Primal-Dual Interior-Point Methods for Domain-Driven Formulations}

\section{Introduction} \label{introduction}
In this article, a \emph{convex optimization} problem is minimizing a \emph{convex function} over a \emph{convex set} in a finite dimensional Euclidean space.
Convex optimization's powerful and elegant theory has been coupled with faster and more reliable numerical linear algebra software and powerful computers to spread its applications over many fields such as (1) data science: machine learning, compressed sensing (see \cite{candes2012exact, hastie2015statistical, donoho2006compressed, ames2011nuclear}), (2) engineering: control theory, signal processing, circuit design (see  \cite{boyd2004convex, ben2001lectures, boyd1994linear, MR1953253}), (3) relaxation and randomization: provable bounds and robust heuristics for hard nonconvex problems (see \cite{tuncel-book}), and (4) robust optimization (see \cite{ben2002robust,ben2009robust}). 
Development of  modern \emph{interior-point methods} has had a huge impact on the popularity of convex optimization. Modern theory of interior-point methods, with polynomial iteration complexity, started with Karmarkar's revolutionary paper \cite{karmarkar} in 1984 and then extended from linear optimization to general convex optimization problems by Nesterov and Nemirovskii \cite{interior-book} in the late 1980's. The literature on this topic has become extensive and many different approaches have been proposed since then.
In this article, we are interested in the modern \emph{primal-dual} interior-point techniques. \cite{cone-free} has a detailed discussion about the advantages of primal-dual techniques over purely primal ones, for example, in designing long-step algorithms. 


The focus of research for primal-dual algorithms has been mostly on \emph{conic formulations} (where minimization is over the intersection of an affine subspace and a \emph{convex cone}), see for example \cite{nesterov1992conic, infea-1, infea-2, nesterov1997self,nesterov1998primal,YTM,MTY93,tunccel2001generalization}. 
Following this research, in many settings, the state-of-the-art for utilizing primal-dual interior-point methods is to reformulate the given convex optimization problem as a conic optimization problem (see \cite{cone-free} or \cite{interior-book}-Section 5.1). This usually requires the introduction of additional variables and constraints (that are artificial in the context of the original problem). However, the applications and software for conic optimization itself have not gone much beyond optimization over \emph{symmetric (self-scaled) cones}; more specifically linear programming (LP), second-order cone programming (SOCP), and semidefinite programming (SDP). Some of the desired properties of optimization over symmetric cones have been extended to general conic optimization \cite{tunccel2001generalization,towards, skajaa-ye,myklebust2014interior}. While the conic reformulation implies that, under reasonable assumptions, all convex optimization problems enjoy the same iteration complexity bounds, there is a gap (remained unchanged for many years) between the efficiency and robustness of the software we have for optimization over symmetric cones and many other classes of problems.  In the \emph{feasible-start}\footnote{Where a pair of points in the relative interior of the primal and dual feasible regions are given.} setup, \cite{cone-free} demonstrated that not all advantages of the primal-dual interior-point techniques are intrinsically related to conic formulation. In this article, we expand this conclusion to the more challenging and practical \emph{infeasible-start} scenario. 
Specifically, we design and analyze infeasible-start primal-dual algorithms for problems given in a typically natural form (can be a conic formulation or not or an arbitrary mixture of both) that not only have comparable theoretical performance to the current best algorithms for conic formulations, but also have been used to create practical software. Let us define our setup:
\begin{definition} \label{def:DD}
\emph{A convex optimization problem is said to be in the \emph{Domain-Driven} setup if it is in the form
\begin{eqnarray} \label{main-p}
\inf _{x} \{\langle c,x \rangle :  Ax \in D\},
\end{eqnarray}
where $x \mapsto Ax$ : $\R^n \rightarrow \R^m$ is a linear embedding, with $A$ and $c \in \mathbb \R^n$ given, and $D \subset \mathbb \R^m$ is a convex set given as the closure of the domain of a $\vartheta$-\emph{self-concordant (s.c.) barrier} $\Phi$.}
\end{definition}
A s.c.\ barrier (rigorously defined in \cite{interior-book} and Appendix \ref{appen-s.c.}) is a convex function whose second derivative regulates its third and first derivatives. 
Every open convex set is the domain of a s.c.\ barrier \cite{interior-book}. Thus, in principle, every convex optimization problem can be treated in the Domain-Driven setup. In applications, the restrictive part of Definition \ref{def:DD}  is that a ``computable"\footnote{Computable means we can evaluate the function and its first and second derivatives at a reasonable cost.} s.c.\ barrier is not necessarily available for a general convex set. However, for many interesting
convex sets (each of which allows us to handle a class of convex optimization problems), we know how to construct an efficient s.c.\ barrier. Specifically,  the feasible region of many classes of problems that arise in practice is the direct sum of small dimensional convex sets with known, computable s.c.\ barriers. In the case of linear programming, for example,   
consider the 1-dimensional set $\{z \in \R: z \geq \beta\}$ for $\beta \in \R$. It is well-known that $-\ln(z-\beta)$ is a s.c.\ barrier for this set. Using this simple function and the fact that if convex sets $D_1$ and $D_2$ have s.c.\ barriers $\Phi_1$ and $\Phi_2$, respectively, then $\Phi_1+\Phi_2$ is a s.c.\ barrier for the direct sum of $D_1$ and $D_2$, we can construct a s.c.\ barrier for any polyhedron; for $A \in \R^{m \times n}$ and $b \in \R^m$, a s.c.\ barrier for
\[
\{x \in \R^n: Ax \leq b\}=\{x \in \R^n: Ax \in D\},
\]
where $D:=b-\R^m_+$,   is $-\sum_{i=1}^m \ln(b_i-a_i^\top x)$, where $a_i^\top$ is the $i$th row of $A$. This discussion for LP exemplifies the fact that knowing a s.c.\ barrier for small dimensional convex sets combined with the direct sum operator lets us solve problems with an arbitrarily large number of variables and constraints (of the same type).  

The power of the Domain-Driven setup is further accentuated when we consider the possibility of direct summing (or alternatively, intersecting) convex sets of different types. 
In the following, we show many set constraints/functions as the building blocks of a problem in the Domain-Driven setup. We start by showing that the Domain-Driven setup covers the popular optimization over symmetric cones. Many of these s.c. functions can be found in Nesterov and Nemirovski's seminal book  \cite{interior-book}.

\noindent {\bf LP, SOCP, and SDP:} optimization over symmetric cones is a special case of the Domain-Driven setup. Table \ref{tbl:DD-example-1} shows the constraints that specify $D$ and a s.c.\ barrier associated with the convex set defined by the constraint. 
\begin{table}  [h]
\centering
  \caption{ LP, SOCP, and SDP constraints and the corresponding s.c.\ barriers. $ \mathbb S^n$ is the set of $n$-by-$n$ symmetric matrices and $A \preceq B$ for $A,B \in \mathbb S^n$ means $B-A$ is positive semidefinite.}
  \label{tbl:DD-example-1}
  \begin{tabular}{ |c | c | c | }
    \hline
     & constraint & {s.c.\ barrier $\Phi$} \\ \hline
    LP& $ z \leq \beta, \ \ z, \beta \in \mathbb R,$  & $-\ln(\beta-z)$ \\ \hline
    SOCP & $ \|z\| \leq t, \ \ z \in \mathbb R^n, \ \ t \in \mathbb R,$ & $-\ln(t^2 - z^\top z)$ \\ \hline
    SDP & $Z \preceq B, \ \ Z,B \in \mathbb S^n$ & $-\ln(\det(B-Z))$ \\ \hline 
  \end{tabular}
\end{table}
For example, if our problem has the constraint $a^\top x \leq \beta$ for $a \in \mathbb R^n, \beta \in \mathbb R$, the convex set defined by this constraint is the set of $x \in \mathbb R^n$ such that $a^\top x \in \{z: z \leq \beta\}$. 

\noindent {\bf Direct sum of 2-dimensional sets:}  The Domain-Driven setup allows inequalities of the form
\begin{eqnarray} \label{intro-3}
\sum_{i=1}^\ell \alpha_i f_i(a_i^\top x + \beta_i) + g^\top x + \gamma  \leq 0,  \ \ \ a_i, g \in \mathbb R^{n}, \ \ \beta_i, \gamma  \in \mathbb R,  \ \ i \in  \{1,\ldots,\ell\},
\end{eqnarray}
where $\alpha_i \geq 0$ and  $f_i(x)$, $i \in  \{1,\ldots,\ell\}$, can be any univariate convex function whose epigraph is a 2-dimensional set equipped with a known s.c.\ barrier. Three popular examples are given in Table \ref{table1}, and several more can be found in \cite{interior-book}. The fact that constraints of the form \eqref{intro-3} fit into the Domain-Driven setup is implied by the following relation:
\begin{eqnarray} \label{eq:DD-example-1}
\begin{array}{rcl}
&&\left \{  x : \sum_{i=1}^\ell \alpha_i f_i(a_i^\top x + \beta_i) + g^\top x + \gamma \leq 0 \right \}   \\
&=&\left \{  x :  \exists u \in \R^\ell \ \text{such that} \ \sum_{i=1}^\ell \alpha_i u_i + g^\top x + \gamma \leq 0, \ \ f_i(a_i^\top x + \beta_i) \leq u_i, \ \forall  i \right \}.
\end{array}
\end{eqnarray}
Note that Geometric Programming \cite{boyd2007tutorial} and Entropy Programming \cite{fang2012entropy} with vast applications in engineering are constructed with constraints of the form \eqref{intro-3} when $f_i(z)=e^z$ for $i\in\{1,\ldots,\ell \}$ and $f_i(z)=z\ln(z)$ for $i\in\{1,\ldots,\ell \}$, respectively.  
\begin{table} [h] 
\centering 
  \caption{ Some 2-dimensional convex sets and their s.c.\ barriers.}
  \label{table1}
  \begin{tabular}{ |c | c | c | }
    \hline
     & set $(z,t)$ & {s.c.\ barrier $\Phi(z,t)$} \\ \hline
    1 & $ e^z \leq t$ & $-\ln(\ln(t)-z)-\ln(t)$ \\ \hline
    2 & $z \ln(z) \leq t, \ z>0$ & $-\ln(t-z\ln(z)) - \ln(z)$ \\ \hline 
    3 & $|z|^p \leq t, \ p \geq 1$ & $-\ln(t^{\frac 2p} - z^2) - 2\ln(t)$ \\ \hline
  \end{tabular}
\end{table} \\
\noindent {\bf Epigraph of matrix norm, minimizing nuclear norm:}  Assume that we have constraints of the form 
\begin{eqnarray} \label{EO2N-1}
&& Z-UU^\top \succeq 0, \ \ \text{where} \ \ Z=Z_0+\sum_{i=1}^\ell x_i Z_i, \ \ U=U_0+\sum_{i=1}^\ell x_i U_i.
\end{eqnarray}
$Z_i$, $i \in  \{0,\ldots,\ell\}$, are $m$-by-$m$ symmetric matrices, and $U_i$, $i \in  \{0,\ldots,\ell\}$, are $m$-by-$n$ matrices. Using the Schur complement theorem, we can reformulate \eqref{EO2N-1} as an SDP constraint with size $m+n$.
However, the set $\{(Z,U): Z-UU^\top \succeq 0 \}$ accepts the following s.c.\ barrier:
\begin{eqnarray} \label{EO2N-3}
\Phi(Z,U):=-\ln(\det(Z-UU^\top)).
\end{eqnarray}
In the cases that $m \ll n$, the parameter of the s.c.\ barrier (responsible for worst-case iteration complexity bounds (see \cite{interior-book} or Appendix \ref{appen-s.c.})) for \eqref{EO2N-3} is much smaller than the one we need for the SDP reformulation, which can make a huge difference both in theory and applications. 

A special application for constraints of the form \eqref{EO2N-1} arises in minimizing the \emph{nuclear norm}. The nuclear norm of a matrix $Z$ is $\|Z\|_*:=\trace\left ((ZZ^\top) ^{1/2} \right)$. The dual norm of $\|\cdot\|_*$ is the 2-norm $\|\cdot\|$ of a matrix. It can be shown that the following optimization problems are a primal-dual pair \cite{recht2010guaranteed}. 
\begin{eqnarray} \label{eq:dual-norm-1}
\begin{array} {ccc}  
(P_N) & \min_{X}  & \|X\|_* \\
& s.t. & A(X)=b.
\end{array} \ \ \ 
\begin{array} {ccc}  
(D_N) & \max_{z}  & \langle b,z \rangle \\
& s.t. & \|A^*(z)\| \leq 1,
\end{array}
\end{eqnarray}
where $A$ is a linear transformation on matrices and $A^*$ is its adjoint. In  machine learning and compressed sensing, $(P_N)$ is a very popular relaxation of the problem of minimizing $\rank(X)$ subject to $A(X)=b$. The dual problem $(D_N)$ is a special case of \eqref{EO2N-1} where $Z=I$ and $U=A^*(z)$. It can be shown that solving $(D_N)$ by our primal-dual techniques immediately gives us a solution for $(P_N)$.  \\
\noindent {\bf Compatibility of s.c.\ barriers, epigraph of quantum entropy and quantum relative entropy:} Another useful theoretical tool for constructing s.c.\ functions and barriers is the \emph{compatibility result}, see Chapter 5 of \cite{interior-book} and Theorem 9.1.1 of  \cite{nemirovski-notes}. Recently, such an approach was used \cite{faybusovich2014matrix,faybusovich2018primal}  to construct a s.c.\ barrier for the epigraph of \emph{quantum entropy}. Consider a function $f: \mathbb R \rightarrow \mathbb R\cup\{+\infty\}$ and let $X\in \mathbb H^n$ be a Hermitian matrix (with entries from $\mathbb C$) with a spectral decomposition $X=U \Diag(\lambda_1,\ldots,\lambda_n) U^*$, where $\Diag$ returns a diagonal matrix with the given entries on its diagonal and $U^*$ is the conjugate transpose of a unitary matrix $U$. Then, $F: \mathbb H^n \rightarrow \R\cup\{+\infty\}$ is defined as 
\begin{eqnarray*}\label{eq:fun_cal_1}
F(X):= \trace(U \Diag(f(\lambda_1),\ldots,f(\lambda_n)) U^*). 
\end{eqnarray*} 
Study of such matrix functions go back to the work of L\" owner as well as Von-Neumann (see \cite{davis1957all}, \cite{lewis2003mathematics}, and the references therein).  
It is proved in \cite{faybusovich2014matrix} (and it follows from the above-mentioned compatibility result) that if $f$ is continuously differentiable with a  \emph{matrix monotone} derivative on $\R_+$, then the function 
\begin{eqnarray*}\label{eq:fun_cal_3}
\Phi(t,X):=-\ln(t-F(X))-\ln \det(X)
\end{eqnarray*}
is a s.c.\ barrier for the epigraph of $F(X)$ in $\mathbb S_+^n$. For $f(x):=x\ln(x)$, the function $F(X)$ is called \emph{quantum entropy}.  In this case, $\Phi(t,X)$ can be seen as a lift for the s.c.\ barrier  we gave in Table \ref{table1} for the entropy function. 
Optimization of quantum entropy and its extension \emph{relative quantum entropy} have many recent applications \cite{chandrasekaran2016relative,chandrasekaran2017relative}. The authors in \cite{fawzi2017relative,fawzi2017semidefinite} approximate these problems by SDP. 
We can handle convex optimization problems involving quantum entropy in the Domain-Driven setup by the above s.c.\ barrier. We also know that $f(t,x,y):=-\ln(t-x\ln(x/y))-\ln(x)-\ln(y)$ is a 3-s.c.\ barrier for the epigraph of the relative entropy \cite{nesterov2006constructing}. We can generalize this to prove that the function $f:\R\oplus\R^n\oplus\R^n \rightarrow \R$ defined as $f(t,x,y):=-\ln(t-\sum_{i=1}^{n}x_i\ln(x_i/y_i))-\sum_{i=1}^n\ln(x_i)-\sum_{i=1}^n\ln(y_i)$ is a $(2n+1)$-s.c.\ barrier for the epigraph of vector relative entropy. Thus, we are able to treat vector relative entropy based convex optimization problems directly in our Domain-Driven setup.  \\
\noindent {\bf Combination of all the above examples:}
Assume that we have $\ell$ convex set constraints in the Domain-Driven form, with corresponding sets $D_1,\ldots,D_\ell$, and corresponding s.c.\ barriers $\Phi_1 , \ldots,\Phi_\ell$. Now, let $D:=D_1 \oplus \cdots  \oplus D_\ell$. Then, $\Phi:=\Phi_1+\cdots+\Phi_\ell$ is a s.c.\ barrier for $D$ \cite{interior-book}, and \eqref{main-p} for this $D$ is also in the Domain-Driven setup. 

\subsection{Contributions of this paper} 
Although the terminology Domain-Driven is new, the concept was proposed in \cite{cone-free}, then named \emph{cone-free}. The underlying algorithms were  feasible-start primal-dual algorithms for problems in the Domain-Driven setup. In theory of convex optimization, having a theory of feasible-start algorithms is sufficient for many purposes. In applications of convex optimization as well as in software, infeasible-start algorithms are essential.  For the infeasible-start setup, the most common approach is  (i) conic reformulation, (ii) using homogeneous self-dual embedding type algorithms (see for example \cite{infea-2}). However, software and applications of modern conic optimization itself has not gone much beyond optimization over symmetric cones.  There are other types of algorithms such as Nesterov and Nemirovski's which approximately follow multi-parameter surfaces of analytic centers \cite{multi}. These algorithms seem too complicated to directly result in a practical code.  

 For infeasible-start algorithms which solve a Newton system at every iteration, we can consider two extremes based on the number of artificial variables. At one extreme (see \cite{lustig1990feasibility,lustig1991computational,kojima1993primal,zhang1994convergence,zhang1998extending}), there is no artificial variable and the systems we solve at every iteration are the same as the ones we solve in the feasible-start case except for a perturbed right-hand-side. In the case of LP, for the primal problem $\min \{c^\top x : Ax=b, x \geq 0\}$ and dual problem $\max \{b^\top y : A^\top y+s=c, s \geq 0\}$, where $A \in \R^{m \times n}$, $c \in \R^n$, and $b \in \R^m$, the system we solve at every iteration is of the form
\begin{eqnarray*} \label{eq:rev-1}
\left[\begin{array}{ccc}A & 0 & 0 \\ 0 & A^\top & I \\  S & 0 & X \end{array}\right] \left[\begin{array}{c} d_x \\ d_y \\ d_s \end{array}\right] = \left[\begin{array}{c} r_p\\ r_d \\ Xs-\mu e \end{array}\right],
\end{eqnarray*}
where $X$ and $S$ are diagonal matrices with $x$ and $s$ on the diagonal, and $r_p:=b-Ax$ and $r_d:=c-s-A^\top y$. If the current point is feasible, $r_p$ and $r_d$ are zero and we get the system for feasible start algorithms. These algorithms work very well in practice and have been very popular since late 1980's (for example used in a once popular code OB1 \cite{lustig1991computational} as well as LIPSOL \cite{lipsol}); however, their complexity analysis has been challenging.  In the case of LP, Kojima, Megiddo, and Mizuno in \cite{kojima1993primal} proved a global convergence result  for a version of these algorithms. Zhang \cite{zhang1994convergence} proved an $O(n^2\ln(1/\epsilon))$ iteration complexity bound for this method, and for some variations the bound was further improved to $O(n\ln(1/\epsilon))$, for example by Mizuno \cite{mizuno1994polynomiality}. 
Recently, these types of algorithms have been used for even non-convex infeasible-start setups \cite{hinder-1,hinder-2}.

At the other extreme are the algorithms which work with a homogeneous self-dual embedding \cite{YTM,infea-2} where we have artificial variables and homogenization variables.  Using this formulation, Ye, Todd, and Mizuno \cite{YTM} achieved the $O(\sqrt{n}\ln(1/\epsilon))$ iteration complexity bound for LP.  Our infeasible-start approach is in the middle, closer to the first group as we add only one artificial variable, but  do not impose an explicit homogenization (moreover, we tie our artificial variable to our central path parameter). Our complexity results here are new for this approach, even in the case of LP where our iteration bound is $O(\sqrt{n}\ln(1/\epsilon))$.

We introduce a notion of duality gap for the Domain-Driven setup and define an infeasible-start primal-dual central path (Section \ref{sec:duality}). 
Then, in Section \ref{sec:alg} we design our path-following algorithms and in Section \ref{sec:analysis} we give the analysis that yields the current best iteration complexity bounds for solving the problem. By solving, we mean determining the status of a given problem (as being unbounded, infeasible, having optimal solutions, etc.) and providing suitable approximate certificates for the status. 
Several cases of ill-conditioning can happen for a given problem. In order to evaluate the performance of any algorithm in determining the status of a problem in the Domain-Driven setup, we need to carefully categorize these statuses \cite{karimi_status_arxiv, karimi_thesis}. In this paper, we briefly discuss how to interpret the outcome of the algorithms and elaborate on the case of strict primal and dual feasibility. The different patterns that can be detected by our algorithms and the iteration complexity bounds for them are comparable to the current best results available for infeasible-start conic optimization, which to the best of our knowledge is mostly in the work of Nesterov-Todd-Ye \cite{infea-2}. The algorithms we design make up the foundation of a new code DDS (Domain-Driven Solver).

Part of the strength and elegance of the interior-point machinery for conic optimization comes from the fact that convex cones accept s.c.\ barriers that are \emph{logarithmically-homogeneous} (LH). Figure \ref{Fig-diagram} shows the relation between various classes of s.c.\ functions. LF conjugate of a LH s.c.\ barrier is also a LH s.c.\ barrier; an important property that we loose for a general s.c.\ barrier. However, importantly, the LF conjugate of a s.c.\ barrier has more properties than an arbitrary s.c.\ function. Another contribution of this article is that in the design and analysis of our algorithms, we vastly exploit this property, which has not been considered at this level of detail in the literature.   
\begin{figure}
\includegraphics[scale=0.55]{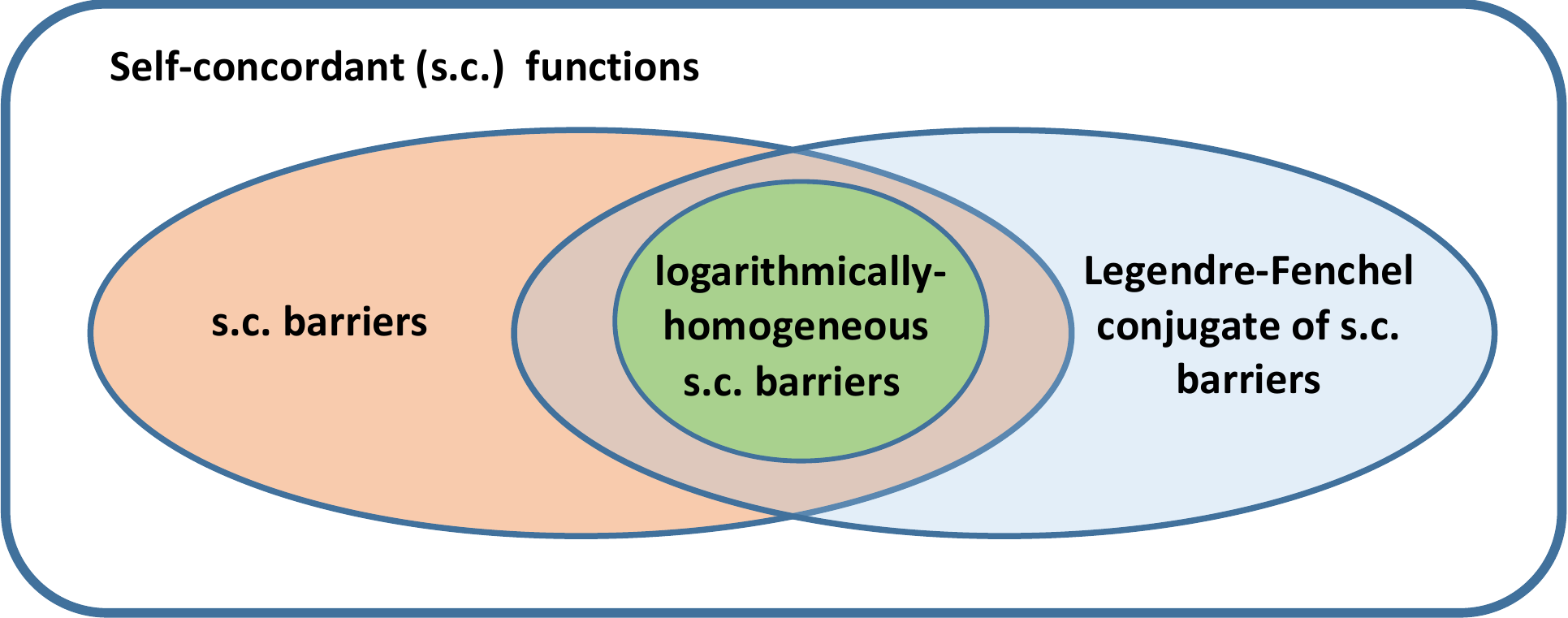}
\centering
\caption{\small A diagram that shows the relationships among various classes of self-concordant functions (see Appendix \ref{appen:examples} for various examples). }
\label{Fig-diagram}
\end{figure}
\subsection{Assumptions and notations} 
 To design our primal-dual algorithms for the Domain-Driven setup, we make some assumptions. First, we assume that the kernel of $A$ is $\{0\}$ in Definition \ref{def:DD}; otherwise we can update $A$ to a matrix $\bar A$ whose columns form a basis for $\img A$ (image of $A$) and then update $c$ and $\Phi$ accordingly (see \cite{interior-book} for stability of s.c.\ barriers under affine maps). We also assume that the Legendre-Fenchel (LF) conjugate $\Phi_*$ of $\Phi$ is given. Even though restricting, such assumptions are unavoidable in the context of primal-dual techniques. Also, for many classes of problems, including the above examples, $\Phi_*$ is computable. The domain of $\Phi_*$ is the interior of a cone $D_*$ defined as (see \cite{interior-book}-Theorem 2.4.2 or \eqref{open-cone}):
 \begin{eqnarray} \label{eq:leg-conj-2} 
D_*=\{y:  \langle y , h \rangle  \leq 0, \ \ \forall h \in \rec(D)\}, 
\end{eqnarray}
where $\rec(D)$ is the recession cone of $D$ (see \eqref{eq:rec-cone}).
Consider an Euclidean vector space $\mathbb E$ with dual space $\mathbb E^*$ and a scalar product $\langle \cdot, \cdot \rangle$.  For a self-adjoint positive definite linear transformation $B: \mathbb E \rightarrow \mathbb E^*$, we define a conjugate pair of Euclidean norms as:
\begin{eqnarray} \label{norms}
\|x\|_B &:=& \left [ \langle Bx,x \rangle \right] ^{1/2},  \nonumber \\
\|s\|^*_B &:=& \max\{\langle s,y \rangle: \ \|y\|_{B} \leq 1\} = \|s\|_{B^{-1}}=  \left [ \langle s, B^{-1}s \rangle \right] ^{1/2}.
\end{eqnarray} 
Note that \eqref{norms} immediately gives us a general Cauchy-Schwarz (CS) inequality:
\begin{eqnarray} \label{eq:CS}
\langle s,x \rangle  \leq \|x\|_B \|s\|^*_B, \ \ \forall x \in \mathbb E, \forall s \in \mathbb E^*. 
\end{eqnarray}
For simplicity, we use abbreviations  RHS and LHS for right-hand-side and left-hand-side, respectively.
We define the following function which is frequently used in the context of self-concordant functions. 
\begin{eqnarray}  \label{eq:rho}
\rho (t) :=\left \{ \begin{array}{ll} t-\ln(1+t)=\frac{t^2}{2}-\frac{t^3}{3}+\frac{t^4}{4}+\cdots,   &  t > -1,   \\ +\infty,   & t \leq -1.\end{array} \right.
\end{eqnarray} 
We also need, in some sense, the inverse of this function 
\begin{eqnarray} \label{eq:sigma-1}
	\sigma(s) := \max\{t: \rho(t) \leq s\}, \ s \geq 0.
\end{eqnarray}
\section{Duality gap for Domain-Driven setup and central path} \label{sec:duality}
Considering the \emph{support function} of $D$,
\begin{eqnarray}  \label{eq:supp-fun-1}
\delta_*(y|D) := \sup\{ \langle y,z \rangle :  z \in D\},
\end{eqnarray}
we define the duality gap as:
\begin{definition} \label{def:duality-gap}
\emph{For every point $ x \in \R^n$ such that $A  x \in D$ and every point $ y \in D_*$ such that $A^\top y = -c$, the \emph{duality gap} is defined as:
\begin{eqnarray} \label{eq:duality-gap-1}
\langle c, x \rangle + \delta_*(y|D).
\end{eqnarray}}
\end{definition}
The following lemma shows that duality gap is well-defined and  zero duality gap is a guarantee for optimality:
\begin{lemma}  \label{lem:dd-7}
	For every point $ x \in \mathbb \R^n$ such that $A  x \in D$ and every point $ y \in D_*$ such that $A^\top y = -c$, we have
	\begin{eqnarray} \label{eq:duality-gap-2}
	\langle c, x \rangle + \delta_*(y|D) \geq 0. 
	\end{eqnarray}
Moreover, if the equality holds above for a pair $(\hat x, \hat y)$ with $A \hat x \in D$ and $\hat y \in D_*, \ A^\top \hat y = -c$, then $\hat x$ is an optimal solution of \eqref{main-p}.
\end{lemma}
\begin{proof} Let $x$ and $ y$ be as above. Then, 
\begin{eqnarray*}
\langle c, x \rangle \underbrace{=}_{A^\top  y=-c} -\langle A^\top  y,  x \rangle = - \langle   y, A  x \rangle	  \underbrace{\geq}_{A x \in D, \  y \in D_*} -  \delta_*(y|D). 
\end{eqnarray*}	
Thus, $\langle c, x \rangle + \delta_*(y|D) \geq 0$, as desired. If equality holds for $(\hat x, \hat y)$, then for every $x$ such that
$Ax \in D$, we have
\[
\langle c, \hat x \rangle  \underbrace{=}_{\text{\eqref{eq:duality-gap-2} holds with equality}}  - \delta_*(\hat y|D)  \underbrace{\leq}_{\text{\eqref{eq:supp-fun-1}}}  - \langle \hat y, Ax \rangle =  \langle - A^\top \hat y, x \rangle   \underbrace{=}_{A^\top \hat y=-c} \langle c,  x \rangle.
\] 
Therefore, $\hat x$ is an optimal solution for \eqref{main-p}.
\end{proof} 
Duality gap must be easily computable and support function is not generally easy to calculate. However, the following theorem shows that we can estimate the support function within any desired accuracy using the fact that $\Phi_*$ is the LF conjugate of a s.c.\ barrier. 
\begin{theorem} [Theorem 2.4.2 of \cite{interior-book}]  \label{thm:support-fun}
Assume that $\Phi$ is a  $\vartheta$-s.c.\ barrier on $D$ and let $\Phi_*$ be the LF conjugate of $\Phi$ with domain $\inte D_*$. Then, for every point $y \in \inte D_*$ we have
\begin{eqnarray} \label{eq-infeas-1}
\delta_*(y|D) - \frac{\vartheta}{k}  \leq \langle \Phi'_*(ky), y \rangle \leq  \delta_*(y|D),\ \ \forall k > 0. 
\end{eqnarray}
Moreover,
\begin{eqnarray} \label{eq:norm-phi*-1}
\Phi''_*(y) [y,y]  \leq \vartheta.
\end{eqnarray}
\end{theorem}
\begin{corollary}
Assume that there exist a sequence $\{z^k\} \in \inte D$ such that $z^k \rightarrow A \hat x \in D$, and a sequence $\{y^k\} \in \inte D_*$ such that $y^k \rightarrow \hat y \in D_*$ and $A^\top \hat y=-c$. If
\[
\lim_{k} \left (\langle c, x^k \rangle + \langle y^k, \Phi'_*(k y^k)  \rangle \right)=0,
\]
then $\hat x$ is an optimal solution of \eqref{main-p}.
\end{corollary}
\begin{proof}
We use Theorem \ref{thm:support-fun} to approximate the support function and then apply Lemma \ref{lem:dd-7}. 
\end{proof}
\subsection{Primal-dual infeasible-start central path}
Our algorithms are \emph{infeasible-start}, which means we do not require a feasible point from the user to start the algorithm. To introduce our infeasible-start central path, we start with a feasible start central path, called cone-free in \cite{cone-free}, which is defined by the set of solutions to:
\begin{eqnarray}  \label{main-f}
\begin{array} {clc} 
    (a)  & Ax \in \inte D, & \\
    (b)  & A^\top y=-\tau c, &   y \in \inte D_*,\\
    (c)  &   y=\Phi'(Ax), &
\end{array}
\end{eqnarray}
where $\tau >0$ is the parameter of the path. It is proved in \cite{cone-free} that under \emph{strict primal-dual feasibility}  (there exists $\hat x$ such that $A \hat x \in \inte D$ and $\hat y \in \inte D_*$ such that $A^\top \hat y = -c$), the system \eqref{main-f} has a unique solution $(x(\tau),y(\tau))$ for every $\tau >0$ and $x(\tau)$ converges to a solution of \eqref{main-p} when $\tau \rightarrow +\infty$. 
Note that we can also prove this by utilizing our notion of duality gap and using Theorem \ref{thm:support-fun}. 

Let us see how to modify \eqref{main-f} for an infeasible-start algorithm. We assume that we can choose a point $z^0 \in \inte D$ and then we define $y^0 := \Phi'(z^0) \in \inte D_*$. We modify the primal and dual feasibility parts of  \eqref{main-f} as follows:
\begin{eqnarray}  \label{eq:mod-path-1}
\begin{array} {clc} 
    (a)  & Ax+\frac{1}{\tau} z^0 \in \inte D, &  \tau > 0, \\
    (b)  & A^\top y=A^\top y^0 - (\tau-1) c, &   y \in \inte D_*,
\end{array}
\end{eqnarray}
where $(x^0:=0,\tau_0:=1,y^0)$, is feasible for this system, and when $\tau \rightarrow +\infty$, we get a pair of primal-dual feasible points in the limit. 
Let us give a name to the set of points that satisfy \eqref{eq:mod-path-1}:
\begin{eqnarray} \label{QDD-copy-2}
 \ \ \ \ \ Q_{DD}:= \left \{ (x,\tau,y): A x + \frac{1}{\tau} z^0 \in \inte D, \ \tau > 0, \ \  A^\top y -A^\top y^0 = -(\tau-1) c, \   y \in \inte D_*  \right \}. 
\end{eqnarray}
Our goal is to design infeasible-start primal-dual algorithms as robust as the best ones for the conic setup, which as far as we know, are the homogeneous self-dual embedding type algorithms proposed in  \cite{infea-2}. 
For the primal-dual conic setup, the duality gap for the modified problem in \cite{infea-2} has the following two crucial properties when the parameter of the path tends to $+\infty$: \\
\noindent{\bf(1)} it tends to zero if the problem is solvable, \\
\noindent{\bf(2)} it tends to $+\infty$ if primal or dual is infeasible. \\
To enforce such a property for the Domain-Driven setup, we treat $\tau$ as a variable (artificial variable) and add another parameter $\mu$ which plays the role of the parameter for the central path. Figure \ref{fig:central-path} schematically shows the primal-dual central paths. 
\begin{figure}
\centering
\includegraphics[scale=0.4]{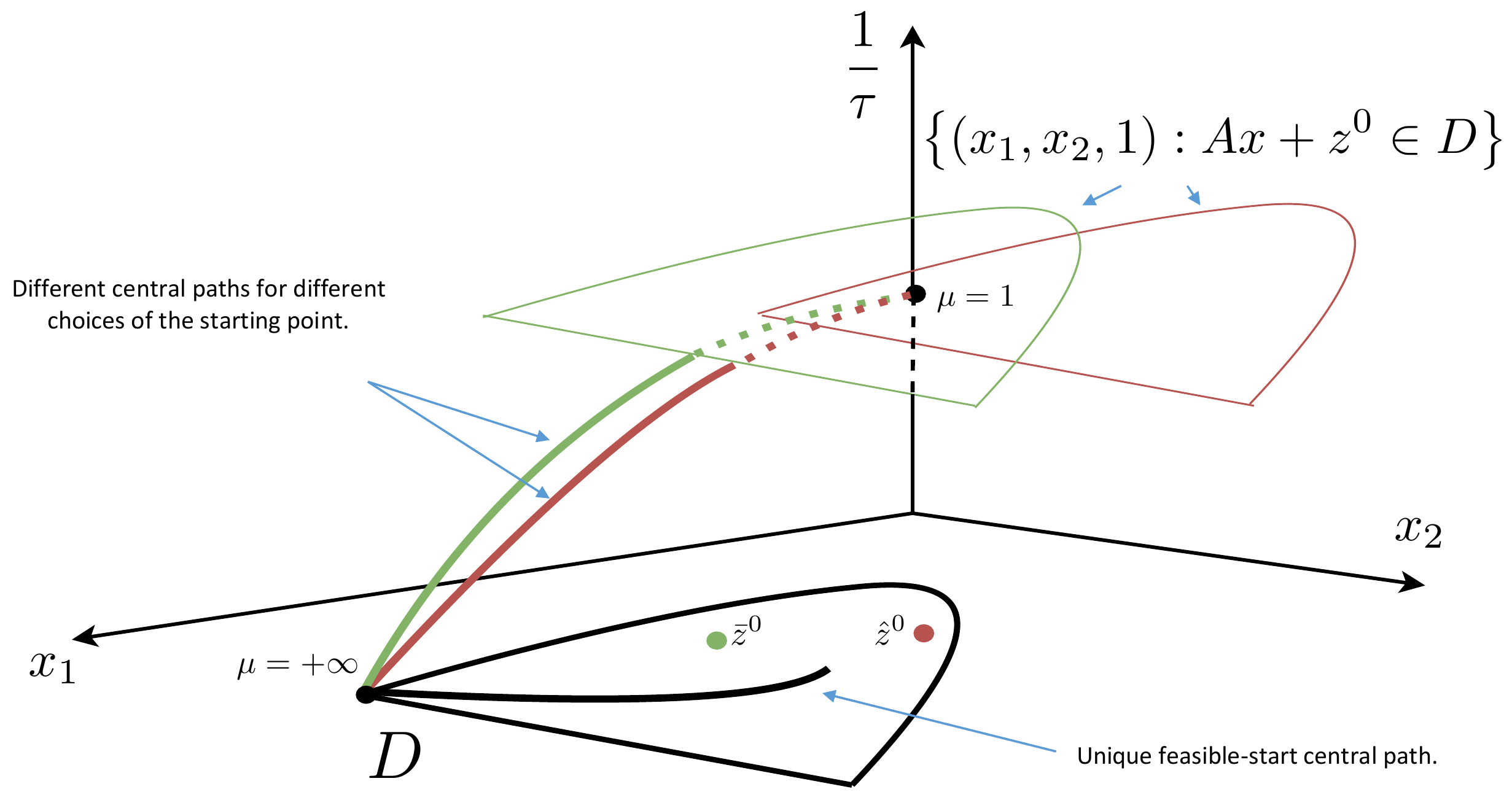}
\caption{\small A problem in the Domain-Driven setup with $D\subset \R^2$, $A:=I$. The infeasible-start and the unique feasible-start central paths projected onto the primal space are shown. }
\label{fig:central-path}
\end{figure}
Let us fix $\xi > 1$ and define:
\begin{eqnarray} \label{starting-points-copy-2}
z^0 := \text{any vector in  $\inte D$}, \ \ y^0 :=  \Phi'(z^0), \ \ y_{\tau,0} := -\langle y^0, z^0 \rangle -\xi \vartheta.
\end{eqnarray} 
\noindent The following theorem defines our central path.
\begin{theorem} \label{thm:cent-path}
Consider the convex set $D \subset \R^m$ equipped with a $\vartheta$-s.c.\ barrier $\Phi$ and let $\Phi_*$ be its LF conjugate with domain $\inte D_*$. Then, for every set of starting points defined in \eqref{starting-points-copy-2}, the system 
\begin{eqnarray} \label{trans-dd-path-1-copy-2}
\begin{array}{rcl}
&(a)&  A x + \frac{1}{\tau} z^0 \in \inte D, \ \ \tau > 0,  \\
&(b)& A^\top y -A^\top y^0 = -(\tau-1) c, \ \ y \in \inte D_*,  \\
&(c)& y=\frac{\mu  }{\tau}  \Phi' \left (  A x + \frac{1}{\tau} z^0 \right), \\
&(d)& \langle c,x \rangle +  \frac{1}{\tau} \langle y, Ax+\frac{1}{\tau} z^0 \rangle =  -\frac{\vartheta \xi \mu}{\tau^2} +\frac{ -y_{\tau,0}}{\tau},
\end{array}
\end{eqnarray}
has a unique solution $(x(\mu), \tau(\mu), y(\mu))$ for every $\mu > 0$.
\end{theorem}
\noindent We denote the solution set of \eqref{trans-dd-path-1-copy-2}  for $\mu > 0$ by the \emph{Domain-Driven primal-dual central path}. Note that for $\mu_0=1$, the point $(x,\tau,y)=(0,1,y^0)$ satisfies all the equations in \eqref{trans-dd-path-1-copy-2}.  
In view of the definition of the central path, for all the points $(x,\tau,y)\in Q_{DD}$, we define
\begin{eqnarray}  \label{eq:dd-4-2}
\begin{array}{rcl}
\mu(x,\tau,y) &:=&   \frac{\tau}{\xi \vartheta}[-y_{\tau,0} - \tau \langle c, x \rangle - \langle y, A x + \frac{1}{\tau} z^0 \rangle], \\
                     &=&- \frac{1}{\xi \vartheta} \left [ \langle y,z^0 \rangle +\tau (y_{\tau,0}+\langle y, Ax \rangle) + \tau^2 \langle c,x \rangle \right] \\
                     &=& - \frac{1}{\xi \vartheta} \left [ \langle y,z^0 \rangle +\tau (y_{\tau,0}+\langle c,x \rangle+ \langle y^0, Ax \rangle) \right], \ \ \ \text{using \eqref{trans-dd-path-1-copy-2}-(b).}
                     \end{array}
\end{eqnarray}
The formula in the second line is a quadratic in terms of $\tau$. However, when we use the dual feasibility condition, we get the third formula that is linear in $\tau$. In other words, the dual feasibility condition removes one of the roots. Assume that both the primal and dual are strictly feasible and we choose $z^0=0$, $x^0$ such that $Ax^0 \in \inte D$, and $y^0$ such that $A^\top y^0 =-c$. Then, the last equation of \eqref{eq:dd-4-2} reduces to $\mu=\tau$ and \eqref{trans-dd-path-1-copy-2} reduces to the cone-free setup in \eqref{main-f}. 
\begin{proof}[Proof of Theorem \ref{thm:cent-path}.]
 Consider the function  $\Phi(\frac{z}{\tau})-\xi \vartheta \ln(\tau)$ that  is a s.c.\ function (see Lemma \ref{lem:dd-9} or the proof of \cite{interior-book}-Proposition 5.1.4). The LF conjugate of this function, as a function of $(y,y_\tau)$, is also a s.c. function \cite{interior-book} and is calculated from the following formula:
\begin{eqnarray}  \label{phi-phi*-copy-2}
\max_{\gamma>0} \left[ \Phi_* (\gamma y)+y_\tau \gamma + \xi \vartheta \ln \gamma \right].
\end{eqnarray}
The gradient of $\Phi(\frac{z}{\tau})-\xi \vartheta \ln(\tau)$ is
\begin{eqnarray} \label{eq:dd-1-2}
 \left[\begin{array}{c} \frac{1}{\tau} \Phi'(\frac{z}{\tau}) \\ -\frac{1}{\tau^2} \langle \Phi'(\frac{z}{\tau}),z \rangle - \frac{ \xi \vartheta}{\tau}\end{array}  \right].
\end{eqnarray}
By substituting  \eqref{trans-dd-path-1-copy-2}-(c) in  \eqref{trans-dd-path-1-copy-2}-(d) and reordering the terms, we can show that for every $\mu >0$, the solution set of \eqref{trans-dd-path-1-copy-2} corresponds to the solution set of the following system
\begin{eqnarray}  \label{eq:dd-2-2}
\begin{array}{rcl}
\left[ \begin{array}{c} y \\ y_\tau \end{array} \right]  &=& \mu \left[\begin{array}{c} \frac{1}{\tau} \Phi'(\frac{z}{\tau}) \\ -\frac{1}{\tau^2} \langle \Phi'(\frac{z}{\tau}),z \rangle - \frac{ \xi \vartheta}{\tau}\end{array}  \right] , \\
z &=&\tau Ax+z^0,  \\
A^\top y &=& A^\top y^0 - (\tau-1)c,  \\
y_\tau &=& y_{\tau,0} + \tau \langle c,x \rangle.
\end{array}
\end{eqnarray}
Consider the following function:
\begin{eqnarray} \label{eq:dd-29}
\begin{array}{rcl}
&& \Phi \left (\frac{z}{\tau} \right)-\xi \vartheta \ln(\tau) + \max_{\gamma>0} \left[ \Phi_* (\gamma y)+y_\tau \gamma + \xi \vartheta \ln \gamma \right]
-\frac{1}{\mu} \left ( \langle y,z \rangle + \tau y_\tau \right)  \\
&\geq&   \Phi \left (\frac{z}{\tau} \right) -\xi\vartheta \ln \tau +  \Phi_* (\gamma y)+y_\tau \gamma + \xi \vartheta \ln \gamma   -    \frac{1}{\mu} \left ( \langle y,z \rangle + \tau y_\tau \right) , \ \ \ \ \forall \gamma >0,
\end{array}
\end{eqnarray}
where the inequality trivially holds because of the $\max$ function.  Let us substitute $\gamma:= \frac{\tau}{\mu}$, then by using the Fenchel-Young inequality (Theorem \ref{thm:FY})
\begin{eqnarray*} \label{eq:dd-30}
\Phi \left ( \frac{z}{\tau}\right) +  \Phi_* \left (\frac{\tau y}{\mu } \right)  \geq  \langle \frac{\tau y}{\mu} , \frac{z}{\tau} \rangle= \frac{1}{\mu} \langle y,z \rangle,
\end{eqnarray*}
we can continue \eqref{eq:dd-29} as
\begin{eqnarray} \label{eq:dd-31}
\geq   \frac{1}{\mu} \langle y , z \rangle  +\frac{1}{\mu} \tau y_\tau - \xi \vartheta \ln \mu -     \frac{1}{\mu} \left ( \langle y,z \rangle +\tau y_\tau \right) 
\  =- \xi \vartheta \ln \mu.
\end{eqnarray}
Hence, the function is bounded from below for every $\mu>0$. Fix $\mu>0$ and consider the optimization problem 
\begin{eqnarray}  \label{eq:dd-2-2-2}
\begin{array}{cc} 
\min & \Phi \left (\frac{z}{\tau} \right)-\xi \vartheta \ln(\tau) + \max_{\gamma>0} \left[ \Phi_* (\gamma y)+y_\tau \gamma + \xi \vartheta \ln \gamma \right]
-\frac{1}{\mu} \left ( \langle y,z \rangle + \tau y_\tau \right)  \\
\textup{s.t.} & Fz =F z^0,\\
& A^\top y = A^\top y^0 - (\tau-1)c,  \\
& y_\tau = y_{\tau,0} +  \langle c_A, z-z^0 \rangle,
\end{array}
\end{eqnarray}
where $F$ is a matrix whose rows form a basis for the kernel of $A^\top$ and $c_A$ is any vector such that $A^\top c_A =c$.
We have $A^\top y = A^\top y^0 - (\tau-1)c$ iff there exists a vector $v$ such that $y=y^0-(\tau-1) c_A - F^\top v$, so over the feasible region of \eqref{eq:dd-2-2-2} we have
\begin{eqnarray}  \label{eq:dd-mu-alt-1}
\begin{array}{rcl}
\langle y,z \rangle + \tau y_\tau &=& \langle y^0-(\tau-1) c_A - F^\top v , z \rangle +  \tau y_{\tau,0} +  \tau \langle c_A, z-z^0 \rangle   \\
 &=& \langle y^0+c_A , z \rangle - \langle v , Fz^0 \rangle + \tau y_{\tau,0}  -\tau  \langle c_A,z^0 \rangle,\\
    &=& \langle y^0+c_A , z \rangle + \langle y , z^0 \rangle + \tau y_{\tau,0} - \langle y^0+ c_A , z^0 \rangle ,
\end{array}
\end{eqnarray}
which is linear in $(z,\tau,y, y_\tau)$. Therefore, the objective function in \eqref{eq:dd-2-2-2} is the summation of a s.c.\ function, its LF conjugate and another term that we showed is linear on the feasible region. Hence, the objective function is a s.c.\ function \cite{interior-book}. Therefore, \eqref{eq:dd-2-2-2}  is minimizing a non-degenerate s.c.\ function that is bounded from below and so attains its unique minimizer $(\bar z, \bar \tau, \bar y, \bar y_\tau)$ \cite{interior-book}. We claim that $(\bar z, \bar \tau, \bar y, \bar y_\tau)$ satisfies the first equality of \eqref{eq:dd-2-2}. Assume that $\frac{1}{\mu} (\hat y,\hat y_\tau)$ is the image of $(\bar z,\bar \tau)$ under the map \eqref{eq:dd-1-2}. Then, we can check that $(\bar z, \bar \tau, \hat y,\hat y_\tau)$ also satisfies the optimality conditions and by uniqueness, we have $(\hat y,\hat y_\tau)=(\bar y,\bar y_\tau)$. To conclude the proof, $F \bar z=F z^0$ implies that there exists a unique $\bar x$ such that $\bar z = \bar \tau A \bar x + z^0$.  Therefore,  $(\bar x, \bar \tau, \bar y)$ is a solution of \eqref{eq:dd-2-2} and so \eqref{trans-dd-path-1-copy-2} for the fixed $\mu$. Uniqueness follows from the fact that, by using Fenchel-Young inequality (Theorem \ref{thm:FY}), the system \eqref{eq:dd-2-2} implies optimality for \eqref{eq:dd-2-2-2}. 
\end{proof}

Let us finish this section by explaining why following the central path defined above solves the problem for us. First we prove the following key lemma:
\begin{lemma}  \label{lem:dg-bound-1}
Let $(x,\tau,y) \in Q_{DD}$, $\mu=\mu(x,\tau,y)$, and $\hat y := \frac{y}{\tau}$. Then,
\begin{eqnarray} \label{eq:dg-bound-1}
	\ \ \ \    \frac{-y_{\tau,0}}{\tau} - \frac{\xi \mu \vartheta+\mu \kappa \sqrt{\vartheta}}{\tau^2} \leq \langle c,x \rangle +\delta_*(\hat y|D)  \leq \frac{-y_{\tau,0}}{\tau} - \frac{(\xi-1)\mu \vartheta-\mu \kappa \sqrt{\vartheta}}{\tau^2},
	\end{eqnarray}
where $\kappa := \left\|Ax+\frac{1}{\tau} z^0-\Phi'_*\left(\frac{\tau}{\mu} y\right)\right\|_{[\Phi''_*(\frac{\tau}{\mu} y)]^{-1}}$.
\end{lemma}
\begin{proof}
By applying Theorem \ref{thm:support-fun} to $k:=\frac{\tau^2}{\mu}$ and $\hat y$ we get
\begin{eqnarray} \label{eq:lem:dg-bound-1-3}
 \langle  \hat y, \Phi'_*\left( \frac{\tau y}{\mu}\right) \rangle \leq \delta_*(\hat y|D) \leq \langle  \hat y, \Phi'_*\left( \frac{\tau y}{\mu}\right) \rangle + \frac{\mu \vartheta }{\tau^2}.
\end{eqnarray}
Note that by adding and subtracting a term we have
\begin{eqnarray} \label{eq:lem:dg-bound-1-4}
 \langle  \hat y, \Phi'_*\left( \frac{\tau y}{\mu}\right) \rangle= \langle  \hat y,  \Phi'_*\left( \frac{\tau y}{\mu}\right) - \left ( Ax+\frac{1}{\tau}z^0 \right) \rangle + \langle   \hat y, Ax+\frac{1}{\tau}z^0  \rangle.
\end{eqnarray}
Now by using the fact that $\|\frac{\tau y}{\mu}\|_{\Phi''_*(\frac{\tau}{\mu} y)}  \leq \sqrt{\vartheta}$ (property \eqref{eq:norm-phi*-1} or \cite{interior-book}-Theorem 2.4.2), definition of $\kappa$, and using CS inequality \eqref{eq:CS} we get
\begin{eqnarray} \label{eq:lem:dg-bound-1-5}
 -\frac{\mu \kappa}{\tau^2} \sqrt{\vartheta} \leq \langle  \hat y,  \Phi'_*\left( \frac{\tau y}{\mu}\right) - \left ( Ax+\frac{1}{\tau}z^0 \right) \rangle \leq \frac{\mu \kappa}{\tau^2} \sqrt{\vartheta}.
\end{eqnarray}
By substituting \eqref{eq:lem:dg-bound-1-5} in \eqref{eq:lem:dg-bound-1-4} and the result in \eqref{eq:lem:dg-bound-1-3} we get
\begin{eqnarray} \label{eq:lem:dg-bound-1-6}
\langle   \hat y, Ax+\frac{1}{\tau}z^0  \rangle - \frac{\mu \kappa}{\tau^2} \sqrt{\vartheta}  \leq \delta_*(\hat y|D) \leq \langle   \hat y, Ax+\frac{1}{\tau}z^0  \rangle + \frac{\mu \kappa}{\tau^2} \sqrt{\vartheta} + \frac{\mu \vartheta }{\tau^2}.
\end{eqnarray}
 By the first line of \eqref{eq:dd-4-2}, we have $\langle   \hat y, Ax+\frac{1}{\tau}z^0  \rangle = -\frac{\xi \vartheta \mu}{\tau^2}-\frac{1}{\tau}y_{\tau,0}- \langle c,x \rangle$. Putting this in  \eqref{eq:lem:dg-bound-1-6} gives us \eqref{eq:dg-bound-1}
\end{proof}

\subsection{Brief interpretation of outcomes of the algorithm} \label{subsec:inter}
  A given problem may have one of several possible statuses. Just in terms of primal feasibility, when the problem is feasible, we can have strict feasibility ($\img A \cap \inte D \neq \emptyset$) or otherwise weak feasibility. When the problem is infeasible, we can have weak infeasibility (an arbitrarily small perturbation makes it feasible), or otherwise strict infeasibility. The same analysis also applies to dual feasibility.  Next, we discuss what can be said in two of the possible cases about the problem  based on the value of $\tau$, using Lemma \ref{lem:dg-bound-1}:
\begin{enumerate}  [(i)]
	\item {\bf $\lim_{\mu \rightarrow +\infty}\tau = +\infty$ such that $\frac{\mu}{\tau^2}$ tends to zero:} Then $x$ converges to a point $\hat x$ that satisfies $A\hat x \in D$ and $\frac{y}{\tau}$ converges to a point $\hat y \in D_*$ that satisfies $A^\top \hat y =-c$. Moreover, Lemma \ref{lem:dg-bound-1} implies that the duality gap $\langle c,\hat x \rangle +\delta_*(\hat y|D)$ is zero. Therefore, by Lemma \ref{lem:dd-7}, $\hat x$ is an optimal solution of the problem.
	
	\item {\bf $\tau$ stays bounded when $\mu \rightarrow +\infty$:} In this case, Lemma \ref{lem:dd-7} shows that $\langle c,x \rangle+\frac{1}{\tau}\delta_*\left(y|D\right)$ tends to $-\infty$. If $\langle c,x \rangle$ tends to $-\infty$, we can argue that the problem is unbounded. If $\langle c,x \rangle$ stays bounded, then Lemma \ref{lem:dd-7} implies that $\bar y :=\lim_{\mu\rightarrow +\infty} \frac{\tau y}{\mu}$ satisfies $\delta_*(\bar y|D) <0$ and we also have $A^\top \bar y=0$. Such a $\bar y$ implies primal infeasibility; otherwise, if there exists $A\bar x \in D$, then we have the following contradiction. 
\[
0 > \delta_*(\bar y|D)   \underbrace{\geq}_{\text{definition of $\delta_*(\bar y|D)$}} \langle \bar  y, A \bar x \rangle = \langle A^\top \bar  y,  \bar x \rangle = 0. 
\]
\end{enumerate} 
Now the question is which statuses make the above cases happen? What is the behavior of $(x,\tau,y)$ when $\mu$ tends to $+\infty$? Answering these questions requires scrutinizing the geometry of the primal and dual problems and a careful categorization of the possible statuses \cite{karimi_status_arxiv, karimi_thesis}. For example, it is proved in \cite{karimi_status_arxiv, karimi_thesis}  that if the problem is strictly primal and dual feasible (and in a weaker sense if it is just solvable), there exists a parameter $\omega>0$ (depending on the geometry of the problem) such that $\tau \geq \omega \mu$ for all the points close to the central path. This implies that the first of the above cases happens when $\mu \rightarrow +\infty$ and our algorithms return an optimal solution. Also different infeasibility and unboundedness statuses are defined in \cite{karimi_status_arxiv, karimi_thesis} and it is shown that in these cases variable $\tau$ is bounded, the second of the above cases happens, and we can extract (approximate and under some conditions exact) certificates of infeasibility or unboundedness out of $(x,\tau,y)$ when $\mu \rightarrow +\infty$. In this general discussion, the ill-conditioned cases (such as both primal and dual are feasible, but the duality gap is not zero) are ignored, which are defined and considered rigorously in \cite{karimi_status_arxiv}. 
\section{Algorithms}\label{sec:alg}
In the previous section, we defined our infeasible-start primal-dual central path, parameterized with $\mu$. In this section, we express a predictor-corrector path-following algorithm that efficiently follows the path to $\mu=+\infty$. 
To define neighborhoods of the central path, we need a notion of proximity. For a point $(x,\tau,y) \in Q_{DD}$, defined in \eqref{QDD-copy-2}, we define a proximity measure as
\begin{eqnarray} \label{DD-proximity-1-1-2}
\Omega_\mu(x,\tau,y) 
                  &:=& \Phi \left ( Ax+\frac{1}{\tau} z^0 \right) +  \Phi_* \left (\frac{\tau y}{\mu } \right) -  \frac{\tau}{\mu} \langle y , A x+ \frac{1}{\tau}z^0 \rangle, \nonumber \\
\mu &:=&\mu(x,\tau,y), \ \ \text{as defined in \eqref{eq:dd-4-2}}.
\end{eqnarray}
Throughout the paper, we may drop the arguments of $\Phi$ and $\Phi_*$ (and also their gradients and Hessians) for simplicity, i.e., $\Phi:=\Phi \left(Ax+\frac{1}{\tau} z^0\right)$ and $\Phi_*:=\Phi_* \left(\frac{\tau}{\mu} y\right)$. 
\begin{remark}
\emph{The proximity measure used for the feasible-start case \cite{cone-free} is 
\begin{eqnarray} \label{eq:cf-proximity}
\Phi(Ax) + \Phi_*(y)-\langle y,Ax \rangle. 
\end{eqnarray}
Even though this proximity measure and \eqref{DD-proximity-1-1-2} have similar structures (indeed for $z^0=0$ and $\tau=\mu$, we recover \eqref{eq:cf-proximity}), $\tau$ and $\mu$ bring nonlinearity into the arguments of $\Phi$ and $\Phi_*$ in \eqref{DD-proximity-1-1-2}. }
\end{remark}
\begin{theorem}  \label{thm:prox-1}
For every $(x,\tau,y) \in Q_{DD}$ and $\mu >0$ we have $\Omega_\mu(x,\tau,y) \geq 0$. Moreover, $\Omega_\mu(x,\tau,y) = 0$ with $\mu=\mu(x,\tau,y)$ iff  $(x,\tau,y)$ is on the central path for parameter $\mu(x,\tau,y)$. 
\end{theorem}
\begin{proof}
Both parts of the theorem are implied by Fenchel-Young inequality (Theorem \ref{thm:FY}) and the definition of the central path. 
\end{proof}
Now, we can state a predictor-corrector algorithm. Note that we choose different step sizes for $x$ and for $(\tau,y)$, i.e., for a search direction $(d_x,d_\tau,d_y)$, the updates are
\begin{eqnarray} \label{eq:update}
x^+:= x+\alpha_1 d_x,   \ \ \ \tau^+:=\tau+\alpha_2 d_\tau,   \ \ \ y^+:= y + \alpha_2 d_y. 
\end{eqnarray}
\noindent\rule{16cm}{0.6pt}\\
\noindent \textbf{Framework for Predictor-Corrector Algorithms} \vspace{-0.3cm} \newline
\noindent\rule{16cm}{0.6pt}\\
\noindent{\bf Input:} $A \in \R^{m \times n}$, $c \in \R^n$, neighborhood parameters  $\delta_1,\delta_2 \in (0,1)$ such that $\delta_1 < \delta_2$, desired tolerance $tol \in (0,1)$. Access to gradient and Hessian oracles for a $\vartheta$-s.c.\ barrier $\Phi$ with $\vartheta \geq 1$ and domain of $\Phi$ equal to $\inte D$. Access to the LF conjugate $\Phi_*$ of $\Phi$, and $z^0 \in \inte D$. \\
\noindent {\bf Initialization:} $k:=0$, $y^0:=\Phi'(z^0)$,  $x^0:=0$, $\tau_0:=1$, and $\mu_0:=\mu(x^0,\tau_0,y^0)$. \\
\noindent {\bf while} (the stopping criteria are not met) 
\begin{addmargin}{1cm}
{\bf if } ($\Omega_{\mu_k}(x^k,\tau_k,y^k) > \delta_1$)
\end{addmargin}
\begin{addmargin}{1.5cm}
Calculate the corrector search direction $(d_x,d_\tau,d_y)$, choose $(\alpha_1,\alpha_2)\in \R^2_+$, and apply the update in \eqref{eq:update} to get  $(x^{k+1},\tau_{k+1},y^{k+1})$, such that $\Omega_{\mu_k}(x^{k+1},\tau_{k+1},y^{k+1})$ is smaller than $\Omega_{\mu_k}(x^k,\tau_k,y^k)$ by a ``large enough" amount. Define  $\mu_{k+1}:=\mu_k$.
\end{addmargin}
\begin{addmargin}{1cm}
{\bf if } ($\Omega_{\mu_k}(x^k,\tau_k,y^k) \leq \delta_1$)
\end{addmargin}
\begin{addmargin}{1.5cm}
 Calculate the predictor search direction $(d_x,d_\tau,d_y)$, choose $(\alpha_1,\alpha_2)\in \R^2_+$, and apply the update in \eqref{eq:update} to get  $(x^{k+1},\tau_{k+1},y^{k+1})$, such that  $\mu_{k+1}:=\mu(x^{k+1},\tau_{k+1},y^{k+1})$ is larger than $\mu_k$ by a ``large enough" amount, while \\
$
\Omega_{\mu_{k+1}}(x^{k+1},\tau_{k+1},y^{k+1})  \leq \delta_2.
$
\end{addmargin}
\begin{addmargin}{1cm}
$k \leftarrow k+1$.
\end{addmargin}
\noindent {\bf end while}\\
\noindent\rule{16cm}{0.6pt}

The best choices for $(\alpha_1,\alpha_2)$ are achieved by a plane search. However, for simplicity of the analysis, it is enough for both the predictor and corrector steps to choose $\alpha_1:=\frac{\alpha_2}{\tau+\alpha_2 d_\tau}$, where $\alpha_2$ is chosen such that $\tau+\alpha_2 d_\tau >0$. Then, our search space becomes 1-dimensional and we can choose $\alpha_2$ large enough to get the desired complexity bounds.

Next, we discuss how to calculate the search directions and choose the step lengths. 
The Dikin ellipsoid property \cite{interior-book} (see Appendix \ref{appen-s.c.}) is perhaps the most fundamental property of s.c.\ functions. This elegant property implies that we can move all the way to the boundary of Dikin ellipsoid and stay feasible. A challenge in our Domain-Driven setup is the nonlinear way that $\tau$ is combined with $x$ and $y$, for example in the proximity measure \eqref{DD-proximity-1-1-2}. What typically appears in a primal-dual proximity measure in the literature is the summation of the s.c.\ barrier and its LF conjugate composed with an affine function of the variables, which makes the algorithm and its analysis easier. The positive definite matrix that defines the Dikin ellipsoid for our algorithm has a special form that controls the nonlinear displacements in the arguments of $\Phi$ and $\Phi_*$ in the proximity measure. 

Let us define $\bar H(x,\tau)$ as  follows (with $u:=Ax+\frac{1}{\tau}z^0$)
\begin{eqnarray} \label{hessian-phi-2}
\ \ \ \ \bar H(x,\tau):= \left [ \begin{array} {cc}  \underbrace{\frac{1}{\tau^2} \Phi'' \left(u \right)}_{=:H}  &   \underbrace{\frac{-1}{\tau^2} \Phi''\left(u\right) u -\frac{1}{\tau^2} \Phi'\left(u\right)}_{=:h} \\  
                   { \left[ \frac{-1}{\tau^2} \Phi''\left(u\right) u -\frac{1}{\tau^2} \Phi'\left(u\right)\right]}^\top   &   \underbrace{\frac{2}{\tau^2} \langle \Phi'\left(u\right) , u \rangle +  
                    \frac{1}{\tau^2} \langle u, \Phi''\left(u\right)  u \rangle + \frac{\xi \vartheta}{\tau^2}}_{=:\zeta} \end{array} \right].
\end{eqnarray}
\begin{remark}  \label{rem-1}
\emph{If we replace $u$ with $\frac{z}{\tau}$ in \eqref{hessian-phi-2}, we get the Hessian for the function $ \Phi \left ( \frac{z}{\tau} \right)+\xi \vartheta \ln(\tau)$. }
\end{remark}
One can easily verify that for every $(d,d_\tau) \in \R^m \oplus \R$ we have
\begin{eqnarray}\label{hessian-phi-22}
\bar H(x,\tau) [(d,d_\tau), (d,d_\tau)]=  \left \|\frac{d}{\tau}-\frac{d_\tau}{\tau} u \right \|^2_{\Phi''(u)} - \frac{2d_\tau}{\tau} \left [\frac{d}{\tau}-\frac{d_\tau}{\tau} u \right ]^\top \Phi'(u)+ \xi \frac{d_\tau^2}{\tau^2} \vartheta.
\end{eqnarray}
By using the definition of s.c barriers \cite{interior-book} (see \eqref{property-2}) for the second term in the RHS of \eqref{hessian-phi-22}, we have
\begin{eqnarray}\label{hessian-phi-2-3}
 \left [ \left \|\frac{d}{\tau}-\frac{d_\tau}{\tau} u \right \|_{\Phi''(u)} - \left|\frac{d_\tau}{\tau}\right| \sqrt{\vartheta}  \right]^2+ (\xi-1) \frac{d_\tau^2}{\tau^2} \vartheta  \leq \bar H(x,\tau) [(d,d_\tau), (d,d_\tau)],
\end{eqnarray}
which shows that $\bar H(x,\tau)$ is a positive definite matrix for every $\xi >1$ and so  invertible. Considering the definition of $H$, $h$, and $\zeta$ in \eqref{hessian-phi-2}, by substitution, one can directly verify that for every $(w,w_\tau) \in \R^m \oplus \R$, we have
\begin{eqnarray} \label{inverse-formula-2-2}
\left [\begin{array}{c}  w  \\  w_\tau  \end{array} \right]^\top \left ( \left [\begin{array} {cc} H & h \\ h^\top & \zeta \end{array} \right] \right ) ^{-1}  \left [\begin{array}{c}  w  \\  w_\tau  \end{array} \right] =  
\langle w, H^{-1}w \rangle + \eta  \left( \langle w, H^{-1}h \rangle - w_\tau \right)^2,  \nonumber \\
H^{-1} h = -u - [\Phi'']^{-1} \Phi', \ \ \ \ 
\eta = \frac{\tau^2}{{\xi \vartheta } -  \langle \Phi', [\Phi'']^{-1} \Phi' \rangle }.
\end{eqnarray}
Note that $\eta\geq 0$ by using $\xi >1$ and $ \langle \Phi', [\Phi'']^{-1} \Phi' \rangle \leq \vartheta$ for $\vartheta$-s.c.\ barriers (see \eqref{property-2-2}). The following key lemma, which we prove later, shows how the spectrum of $\bar H(x,\tau)$ is bounded close to the central path. 
\begin{lemma}  \label{lem:hessian-close-1}
	For every  $\bar \epsilon \in (0,1)$, there exist $\epsilon >0$ depending on $\xi$ such that for every pair $(x,\tau,y) \in Q_{DD}$ and $\mu>0$ with $\Omega_\mu(x,\tau,y) \leq \epsilon$, we have
	\begin{eqnarray}  \label{eq:lem:hessian-close-1-1}
	(1-\bar \epsilon)^2   \bar H(x(\mu),\tau(\mu))   \preceq  \bar H(x,\tau)     \preceq \frac{1}{(1-\bar \epsilon)^2}   \bar H(x(\mu),\tau(\mu)).
	\end{eqnarray}
\end{lemma}
Let $F$ be a matrix whose rows give a basis for the kernel of $A^\top$ and let $c_A$ be any vector such that $A^\top c_A=c$.  We define a block matrix $U$ that comes up frequently in our discussion and contains the linear transformations we need, and also the vector $r^0$ that is used in the RHS of our systems: 
\begin{eqnarray}  \label{eq:mat-U}
U:=\left [\begin{array}{ccc}  
A & 0 & 0 \\ 0 & 1 & 0 \\ 0 & -c_A & -F^\top \\ c^\top & 0 & 0 \end{array} \right], \ \ \ \
r^0:= \left[\begin{array}{c} -A^\top y^0 -  c \\ -y_{\tau,0}+ \langle c_A, z^0\rangle \\ Fz^0 \end{array} \right]. 
\end{eqnarray}
At a current point $(x,\tau,y)$, both the predictor and corrector steps are derived by solving the system 
\begin{eqnarray} \label{eq:system-pred-1}
U^\top  \underbrace{ \left[ \begin{array}{cc} \bar H(x,\tau)  &  0  \\  0  &  \left [\hat H(x,\tau,y) \right] ^{-1} \end{array} \right]}_{\cH(\bar H, \hat H)} U  \left [ \begin{array}{c}  \bar d_x \\ d_\tau \\ d_v   \end{array}\right]=r_{RHS}, 
\nonumber \\
d_x:=\bar d_x-d_\tau x,  \ \ \ \ \   d_y:=-d_\tau c_A - F^\top d_v,
\end{eqnarray}
where $\hat H(x,\tau,y)$ is a positive definite matrix that we elaborate more on later. For both the predictor and corrector steps, we discuss the choice of $\hat H(x,\tau,y)$ and $r_{RHS}$ in \eqref{eq:system-pred-1}.
Before discussing the computational aspects of system \eqref{eq:system-pred-1}, let us elaborate more on the LHS matrix of this system. In both the predictor and corrector steps, we have a vector $d^\top:= [\bar d_x^\top \  \ d_\tau \ \ d_v^\top]$ as the solution of \eqref{eq:system-pred-1} which satisfies $d^\top U^\top \cH(\bar H, \mu^2 \bar H) U d \leq q$ for $\cH$ defined in \eqref{eq:system-pred-1} and a scalar $q$ which takes different values in our analyses.
Let us define
\begin{eqnarray*} \label{eq:lem:ana-1-4}
	f:=Ud=U \left [ \begin{array}{c} \bar d_x  \\ d_\tau \\ d_v   \end{array}\right] \underbrace{=}_{\text{\eqref{eq:mat-U}}} \left[ \begin{array}{c} A \bar d_x  \\ d_\tau \\ -d_\tau c_A-F^\top d_v \\ \langle c, \bar d_x \rangle   \end{array}\right] \underbrace{=}_{\text{\eqref{eq:system-pred-1}}}\left[ \begin{array}{c} A \bar d_x  \\ d_\tau \\ d_y \\ \langle c, \bar d_x \rangle   \end{array}\right].
\end{eqnarray*}
Using \eqref{hessian-phi-2-3} and \eqref{inverse-formula-2-2}, $f^\top  \cH(\bar H, \mu^2 \bar H) f \leq q$ yields
\begin{eqnarray}  \label{eq:lem:ana-1-1}
&& \left [ \left \|\frac{A\bar d_x}{\tau}-\frac{d_\tau}{\tau} \left ( Ax+\frac{1}{\tau} z^0 \right) \right \|_{\Phi''} - \left|\frac{d_\tau}{\tau}\right| \sqrt{\vartheta}  \right]^2+ (\xi-1) \frac{d_\tau^2}{\tau^2} \vartheta   \nonumber \\
&+& \frac{\tau^2}{\mu^2} \langle d_y, [\Phi'']^{-1}d_y \rangle + \frac{\left[\langle \frac{\tau d_y}{\mu},[\Phi'']^{-1} \Phi' \rangle + \frac{\tau}{\mu} (\langle d_y, Ax+\frac{1}{\tau} z^0 \rangle + \langle c, \bar d_x \rangle)\right]^2}{{\xi \vartheta } -  \langle \Phi', [\Phi'']^{-1} \Phi' \rangle }  \leq q.
\end{eqnarray}
There are four nonnegative terms in the LHS of \eqref{eq:lem:ana-1-1} all bounded by $q$ which we will break down to extract the required bounds for our analyses. In other words, the choice of matrices in \eqref{eq:system-pred-1} gives us Dikin ellipsoid type bounds for $d_x$, $d_y$ and $d_\tau$.

\begin{remark} \label{computation}
In the Domain-Driven setup, $A \in \R^{m \times n}$, with $m \geq n$, is given which may have some computationally useful properties such as a special sparsity pattern. In the theoretical formula of \eqref{eq:system-pred-1}, having $F$ is fine, whereas  in practice, calculating $F$ can be very costly and most likely $F$ does not maintain the structure of $A$. However, it can be shown that $F$ is not needed for solving system \eqref{eq:system-pred-1} and it can be eliminated by using the fact that the kernel of $F$ is the range of $A$.  With this observation, system \eqref{eq:system-pred-1} is reduced to solving two systems with LHS matrices $\tilde A^\top \bar H(x,\tau) \tilde A$ and $\tilde A^\top \hat H(x,\tau,y) \tilde A$, where $\tilde A \in \R^{(m+1) \times (n+1)}$ is a matrix constructed by $A$ and $c$. 
This is the same as the general algorithm in \cite{infea-2} for conic optimization. However, as mentioned in \cite{infea-2}, for symmetric cones, equipped with \emph{self-scaled} barriers, the search direction can be calculated by one system of the same size, with an extra cost of calculating a scaling point in the primal cone (or the dual cone). As we mention in the following, one choice for $\hat H(x,\tau,y)$ to achieve the desired theoretical results is $\hat H(x,\tau,y) := \mu^2 \bar H(x,\tau)$, by which the above two systems become the same. A disadvantage of this choice is that it does not fully exploit the  primal-dual symmetry. 

Assume that in the conic setup, where the underlying LH-s.c.\ barriers are $F$ and $F_*$, the current iteration is at primal point $x \in \inte K$ and dual point $s \in \inte K_*$. A key property of symmetric cones and self-scaled barriers is the existence  of a \emph{scaling point} $w \in \inte K$ which satisfies (among other conditions):
\begin{eqnarray}\label{eq:scaling-1}
\begin{array}{crcl}
(a)& F''(w) x &=&s, \\ 
(b)& F''(w) F_*'(s)& =& F'(x). 
\end{array}
\end{eqnarray}
The existence of a scaling point $w \in \inte K$ satisfying \eqref{eq:scaling-1} is not guaranteed for every pair $(x,s) \in \inte K \oplus \inte K_*$, if $F$ is not self-scaled (i.e., beyond symmetric cones), as shown in \cite{myklebust2014interior}. The second author had proposed replacing $F''(w)$ by a positive definite symmetric bilinear form satisfying \eqref{eq:scaling-1} and another useful condition involving $F''(x)$ and $F_*''(s)$ \cite{tunccel2001generalization}. Since \cite{myklebust2014interior, tunccel2001generalization} use low-rank updates to modify the scaling map, the work per iteration can be made comparable to that of a first-order method. 
We mention two other approaches for non-symmetric cones. Nesterov \cite{towards} designed an algorithm that at every iteration, computes a point $(x,s) \in \inte K \oplus \inte K_*$ and a scaling point $w \in \inte K$ which satisfy \eqref{eq:scaling-1}-(a) exactly and \eqref{eq:scaling-1}-(b) approximately. To find such  points, the corrector operation is done only on the primal part to find a point $u$, and $(x,s)$ is a \emph{primal-dual  lifting} of $u$. The approach does not fully exploit the properties of the dual point. Another disadvantage is that the approach is feasible-start, but still requires a phase-I for finding an initial point close to the primal central path. Skajaa an Ye \cite{skajaa-ye} also addressed this issue for non-symmetric cones by designing an algorithm which only uses the primal barrier. Implied by our discussion in Remark \ref{remark-cone-1}, their predictor direction is a special case of ours when we choose $\hat H(x,\tau,y) := \mu^2 \bar H(x,\tau)$. In the corrector step, they propose a quasi-Newton type approach by using low-rank updates for the LHS matrix and so reducing the cost of factoring this matrix at each corrector iteration. 

Overall, our algorithms can be implemented by solving one linear system of equations of size roughly $n$-by-$n$ at every iteration, and the amount of work to form this system depends on practical considerations. Quasi-Newton type updates of the form in \cite{myklebust2014interior, tunccel2001generalization} are the most promising ones we can adapt which can make our algorithms scalable, while attaining some primal-dual symmetry. 
\end{remark}
\subsection{Predictor step} 
An efficient predictor search direction must increase $\mu$ by a large rate and at the same time let us take a long enough step. We first give the choices of $\hat H(x,\tau,y)$ and $r_{RHS}$ for the system in \eqref{eq:system-pred-1} and then justify them. 
For the RHS vector we choose $r_{RHS}:=r^0/{\mu^2}$, where $r_0$ defined in \eqref{eq:mat-U}. We have different choices for $\hat H(x,\tau,y)$  to attain our desired properties (such as a low complexity bound). We express a sufficient condition and discuss two choices that satisfy the condition.  
 We will see that to achieve enough increase in $\mu$ at every predictor step, it is sufficient that for every $\bar \epsilon \in (0,1)$, there exists a choice of $\epsilon$ in Lemma \ref{lem:hessian-close-1} such that
\begin{eqnarray}  \label{eq:lem:hessian-close-1-2}
 (1-\bar \epsilon)^2   \left [\bar H(x(\mu),\tau(\mu))\right]^{-1}   \preceq  \mu^2 \left [\hat H(x,\tau,y)\right]^{-1}     \preceq \frac{1 }{(1-\bar \epsilon)^2}   \left [\bar H(x(\mu),\tau(\mu)) \right]^{-1},
\end{eqnarray}
for every point $(x,\tau,y) \in Q_{DD}$ with $\Omega_\mu(x,\tau,y) \leq \epsilon$.
%

\begin{remark}\label{remark-cone-1} \emph{In view of \eqref{eq:lem:hessian-close-1-1}, one obvious choice for $\hat H(x,\tau,y)$ is 
$\hat H(x,\tau,y) := \mu^2 \bar H(x,\tau)$. 
Another choice is one that yields the predictor direction for the primal-dual conic setup given in \cite{infea-1,infea-2}. 
More explicitly, assume that we reformulate our problem in the Domain-Driven setup as a conic optimization problem by adding an artificial variable, with the conic hull of $D$ as the underlying cone (see \cite{cone-free} or \cite{interior-book}-Section 5.1). By using $\Phi$ and $\Phi_*$, we can construct $\hat \vartheta $-LH-s.c.\ barriers $\hat \Phi$ and  $\hat \Phi_*$, where calculating $\hat \Phi_*$ requires a one-dimensional maximization \cite{cone-free}.  Then, the predictor step calculated in \cite{infea-1,infea-2} can be achieved by \eqref{eq:mat-U} for a special choice of $\hat H$:
\begin{eqnarray} \label{eq:dual-4}
&&\left [\hat H(x,\tau,y) \right]^{-1}:= \left[ \begin{array}{cc} G+\eta_* h_* h_*^\top  &  -\eta_* h_*  \\  -\eta_* h_*^\top  & \eta_*  \end{array} \right], \nonumber \\
&&G:= \bar \tau^2 \Phi''_*(\bar \tau y), \ \ h_*:=- \Phi'_*(\bar \tau y) - \bar \tau \Phi''_*(\bar \tau y)y, \ \ 1/\eta_*:=\frac{\xi \vartheta}{ \bar \tau^2} - \langle y, \Phi''_*(\bar \tau y)y \rangle,
\end{eqnarray} 
where 
\begin{eqnarray} \label{eq:dual-1}
\bar \tau:= \textup{argmax}_{\tau} \left\{ \Phi_*(\tau y) + y_\tau  \tau + \xi \vartheta \ln \tau\right\}.
\end{eqnarray}               
To check that \eqref{eq:dual-4} satisfies condition \eqref{eq:lem:hessian-close-1-2} for every $\bar \epsilon \in (0,1)$ for a right choice of $\epsilon$, we can use the arguments in \cite{infea-1,infea-2}, or the fact that $ \Phi \left ( \frac{z}{\tau} \right)+\xi \vartheta \ln(\tau)$ is a s.c.\ function, Lemma \ref{lem:hessian-close-1}, and the properties of LF conjugates.  Calculating $\bar \tau$ can be done efficiently, since evaluating the RHS of \eqref{eq:dual-1} is equivalent to minimizing a s.c.\ function.   }\end{remark}

Let us justify our predictor step. If we choose $\alpha_1=\frac{\alpha_2}{\tau+\alpha_2 d_\tau}$ (assuming $\tau+\alpha_2 d_\tau >0$) for the updates in \eqref{eq:update}, then by using the third line of \eqref{eq:dd-4-2} for $\mu$, we have
\begin{eqnarray} \label{eq:pred-3}
\begin{array}{cl}
& \mu(x^+,\tau^+,y^+)-\mu(x,\tau,y)  \\
=&  \frac{-1}{\xi \vartheta} \left[ \alpha_2 \langle d_y,z^0 \rangle + \alpha_2 d_\tau y_{\tau,0} + \langle c+A^\top y^0, \alpha_2 d_\tau x + (\tau+\alpha_2 d_\tau) \alpha_1 d_x \rangle \right] \\  
=&  \frac{- \alpha_2}{\xi \vartheta} \left( \langle d_y,z^0 \rangle+ d_\tau y_{\tau,0} +\langle c+A^\top y^0, d_\tau x+ d_x \rangle \right),  \hfill \text{substituting $\alpha_1=\frac{\alpha_2}{\tau+\alpha_2 d_\tau}$,}  \\
=& \frac{\alpha_2}{\xi \vartheta} \left( \langle d_v,Fz^0 \rangle+ d_\tau (\langle c_A, z^0 \rangle-y_{\tau,0}) -\langle c+A^\top y^0, \bar d_x \rangle \right),  \ \ \hfill   \text{substituting $d_x$ and $d_y$ from \eqref{eq:system-pred-1},}  \\
=& \frac{\alpha_2}{\xi \vartheta} [\bar d_x^\top \ \ \ d_\tau \ \ \ d_v^\top]  r^0, \hfill \text{for $r^0$ defined in \eqref{eq:mat-U}}. 
\end{array} 
\end{eqnarray}
Let $d^\top:= [\bar d_x^\top \ \ d_\tau \ \ d_v^\top]$, then we see that the Dikin ellipsoid type constraint  $d^\top U^\top \cH(\bar H, \hat H) U d  \leq 1$ guarantees the feasibility of new iterates with respect to the domains of the underlying s.c.\ functions. 
The search direction in \eqref{eq:system-pred-1} is, up to some scaling, the solution of the following optimization problem
\begin{eqnarray} \label{eq:pred-1}
\begin{array}  {cc}
\max &  \langle d,r^0 \rangle \\
s.t.  &  d^\top U^\top \cH(\bar H, \hat H) U d  \leq 1,
\end{array}
\end{eqnarray}
which can be seen as maximizing the linear function of \eqref{eq:pred-3} in a trust region. 
\subsection{Corrector step} 
After doing a predictor step to increase $\mu$, we need to perform corrector steps to come back into the small neighborhood. Note that our proximity measure $\Omega_\mu(x,\tau,y)$ is not a convex function and to decrease it we use a quasi-Newton like step. In most of the literature on this topic, for example papers \cite{infea-1,infea-2,cone-free}, the corrector step is simply minimizing a s.c.\ function that can be done efficiently by taking damped Newton steps \cite{interior-book}. Even though our proximity measure is not a s.c.\ function and we cannot directly use damped Newton steps, $\Phi$ and $\Phi_*$ are 1-s.c.\ functions and we can exploit their properties. 
We first define the corrector step and then explain our choice. The corrector search direction is the solution of \eqref{eq:system-pred-1} with 
\begin{eqnarray} \label{eq:system-corrector-step-1}
\hat H:= \mu^2 \bar H, \ \ \ r_{RHS}:= -(U^\top \psi ^c+ \beta r^0), \ \ \ \beta:= -\frac{\langle r^0, [U^\top  \cH(\bar H, \mu^2 \bar H) U ]^{-1} U^\top \psi ^c \rangle}{ \langle r^0, [U^\top  \cH(\bar H, \mu^2 \bar H) U ]^{-1} r^0 \rangle},
\end{eqnarray}
where  $r_0$ is defined in \eqref{eq:mat-U} and
\begin{eqnarray} \label{eq:corrector-10-2}
\psi ^c:=\left[ \begin{array}{c}\frac{1}{\tau} \Phi'  \\ -\frac{1}{\tau} \langle\Phi' ,Ax+\frac{1}{\tau}z^0 \rangle +\frac{1}{\mu} \langle y, \Phi_*'  \rangle + \frac{1}{\mu} (y_{\tau,0}+\tau \langle c,x \rangle)\\  \frac{ \tau}{\mu} \Phi_*'   \\ \frac{\tau}{\mu} \end{array}\right].
\end{eqnarray}
\begin{remark}    \label{rem-2}
\emph{If we choose $\alpha_1=\frac{\alpha_2}{\tau+\alpha_2 d_\tau}$ (assuming $\tau+\alpha_2 d_\tau >0$) for the updates in \eqref{eq:update}, then \eqref{eq:pred-3} holds. The parameters in \eqref{eq:system-corrector-step-1} are chosen so that the solution of \eqref{eq:system-pred-1} satisfies $[\bar d_x^\top \ \ d_\tau \ \ d_v^\top]  r^0=0$ and thus, we automatically have $\mu(x^+,\tau^+,y^+)=\mu(x,\tau,y)$ in the corrector step.}
\end{remark}
The following lemma justifies our corrector search direction. 
\begin{lemma}  \label{lem:corr-3}
Consider a choice of $r_{RHS}$ and $\hat H(x,\tau,y)$ such that for the solution of \eqref{eq:system-pred-1} and the updates in \eqref{eq:update} with $\alpha_1=\frac{\alpha_2}{\tau+\alpha_2 d_\tau}$ we have $\mu(x^+,\tau^+,y^+)=\mu(x,\tau,y)$. Then,
\begin{eqnarray}\label{eq:corrector-16-2}
\begin{array}{ccl}
 \rho(D(\alpha_2))  & \leq &\Omega_\mu(x^+ ,\tau^+,y^+)-\Omega_\mu(x,\tau,y) -\alpha_2   \left [ \bar d_x^\top \ \ d_\tau \  \  d^\top_v \right]  U^\top  \psi^c  \\
&&+ \frac{\alpha_2^2 d_\tau}{\tau(\tau+ \alpha_2 d_\tau)}  \langle \Phi', A \bar d_x -d_\tau \left(Ax+\frac 1\tau z^0 \right) \rangle - \frac{\alpha_2^2 d_\tau}{\mu}  \left ( \langle d_y,\Phi_*' \rangle + \langle c, \bar d_x \rangle \right)   \\
&\leq &\rho \left(-D(\alpha_2) \right),  \\
D(\alpha_2)&:=& \frac{\alpha_2}{\tau+\alpha_2d_\tau}\left \| A \bar d_x -d_\tau \left(Ax+\frac 1\tau z^0 \right) \right\|_{\Phi''}+\alpha_2\left \| \frac{d_\tau y + (\tau +\alpha_2 d_\tau) d_y}{\mu} \right\|_{\Phi_*''},
\end{array} 
\end{eqnarray}
where $\psi^c$ is defined in \eqref{eq:corrector-10-2} and $\rho$ is defined in \eqref{eq:rho}. 
\end{lemma}
\begin{proof}
To derive \eqref{eq:corrector-16-2}, we substitute for $\Omega_\mu$ from \eqref{DD-proximity-1-1-2} and then use \cite{nemirovski-notes}-(2.4) (the bounds \eqref{property-4}) for the 1-s.c. function $f(u,w)=\Phi(u)+\Phi_*(w)$. We just need to explicitly calculate the displacements in the arguments of $\Phi$ and $\Phi_*$. 
By the hypothesis, $\mu^+:=\mu(x^+,\tau^+,y^+)=\mu$.
First we have
\begin{eqnarray} \label{eq:dd-expand-4}
\begin{array}{rcl}
Ax^++\frac{1}{\tau^+} z^0 - Ax-\frac{1}{\tau}z^0 &=& \alpha_1 Ad_x- \frac{\alpha_2 d_\tau}{\tau(\tau+\alpha_2 d_\tau)}  z^0  \\
   &=& \frac{\alpha_2}{\tau+\alpha_2 d_\tau}  \left[ A  d_x-\frac{d_\tau}{\tau}z^0 \right],  \hfill  \text{using $\alpha_1=\frac{\alpha_2}{\tau+\alpha_2 d_\tau}$,} \\
   &=& \frac{\alpha_2}{\tau+\alpha_2 d_\tau}  \left[ A \bar d_x-d_\tau \left (Ax+\frac{1}{\tau}z^0\right) \right],  \ \ \hfill \text{using $d_x=\bar d_x-d_\tau x$}.
   \end{array}
\end{eqnarray}
For displacement in the argument of $\Phi_*$, we have
\begin{eqnarray*}
\frac{\tau^+y^+}{\mu^+}-\frac{\tau y}{\mu}= \frac{\alpha_2 d_\tau y + \alpha_2 \tau d_y + \alpha_2^2 d_\tau d_y}{\mu}.
\end{eqnarray*}
As an intermediate step, similar to \eqref{eq:dd-expand-4}, by substituting $\alpha_1=\frac{\alpha_2}{\tau+\alpha_2 d_\tau}$ and $d_x=\bar d_x-d_\tau x$, we have
\begin{eqnarray}  \label{eq:interm-1}
\tau^+ x^+=(\tau+\alpha_2 d_\tau)(x+\alpha_1 d_x)=\tau x + \alpha_2 \bar d_x.
\end{eqnarray}
Then, by using $\mu^+=\mu$ and the first line of \eqref{eq:dd-4-2}, and then substituting \eqref{eq:interm-1}, we have
\begin{eqnarray*}
\begin{array}{rcl}
- \langle \frac{\tau^+y^+}{\mu^+}, Ax^++\frac{1}{\tau^+}z^0 \rangle+ \langle \frac{\tau y}{\mu}, Ax + \frac{1}{\tau} z^0 \rangle 
  &=& - \frac{\tau^+}{\mu }[-y_{\tau,0} - \tau^+ \langle c, x^+ \rangle] +  \frac{\tau}{\mu }[-y_{\tau,0} - \tau \langle c, x \rangle] \\
  &=& \frac{\alpha_2}{\mu} [ d_\tau y_{\tau,0} +\tau \langle c, \bar d_x\rangle+ d_\tau \langle c, \tau x+\alpha_2 \bar d_x \rangle ].
  \end{array}
\end{eqnarray*}
 We can verify by direct substitution that 
\begin{eqnarray} \label{eq:corrector-10-2-T1}
U ^\top \psi ^c=\left[ \begin{array}{c}\frac{1}{\tau} A^\top \Phi' +\frac{1}{\mu} \tau c  \\ -\frac{1}{\tau} \langle\Phi' ,Ax+\frac{1}{\tau}z^0 \rangle +\frac{1}{\mu} \langle y-\tau c_A,  \Phi_*'  \rangle + \frac{1}{\mu} (y_{\tau,0}+\tau \langle c,x \rangle)\\  -\frac{ \tau}{\mu} F \Phi_*'  \end{array}\right].
\end{eqnarray}
If we also use the equality $F^\top d_v=-d_y-d_\tau c_A$, then we have
\begin{eqnarray} \label{eq:corrector-10-2-T2}
\left [ \bar d_x^\top \ \ d_\tau \  \  d^\top_v \right]  U ^\top \psi ^c &=&\frac{1}{\tau}  \langle \Phi', A \bar d_x -d_\tau \left(Ax+\frac 1\tau z^0 \right) \rangle + \frac{\tau}{\mu} \langle c,\bar d_x \rangle   \nonumber \\
&& +\frac{d_\tau}{\mu} \langle y, \Phi_*' \rangle + \frac{\tau}{\mu} \langle d_y,  \Phi_*' \rangle + \frac{d_\tau}{\mu} (y_{\tau,0}+\tau \langle c,x \rangle). 
\end{eqnarray}
By substituting all the above equations we get \eqref{eq:corrector-16-2}. 
\end{proof}
In view of \eqref{eq:pred-3}, $\mu^+=\mu$ is equivalent to $\langle d, r^0 \rangle =0$ for $d^\top :=  \left [ \bar d_x^\top \ \ d_\tau \  \  d^\top_v \right] $. The corrector search direction in \eqref{eq:system-corrector-step-1} is, up to some scaling, the optimal solution of
\begin{eqnarray} \label{eq:lem:corr-ana-1}
\begin{array}  {cc}
\min &  \langle d, U^\top \psi ^c \rangle \\[.3em]
s.t.   &  \langle d, r^0 \rangle =0 \\[.3em]
  &  d^\top U^\top \cH(\bar H, \mu^2 \bar H) U d  \leq 1. 
\end{array}
\end{eqnarray}
Before a concrete analysis, to intuitively justify this search direction using \eqref{eq:corrector-16-2}, note that our goal is to minimize $\Omega_\mu(x^+ ,\tau^+,y^+)-\Omega_\mu(x,\tau,y)$. The coefficient of $\langle d, U^\top \psi ^c \rangle$ is $\alpha_2$, whereas all the other terms are (almost) proportional to $\alpha^2_2$. Therefore, we can look at $\alpha_2 \langle d, U^\top \psi ^c \rangle$ as the first order approximation of $\Omega_\mu(x^+ ,\tau^+,y^+)-\Omega_\mu(x,\tau,y)$ that we minimize in \eqref{eq:lem:corr-ana-1} in a trust region. 
\begin{remark}
\emph{What we prove for the corrector step above is enough for the purposes of obtaining the desired complexity results. However, corrector steps in most of the other papers in this context (such as \cite{infea-1,infea-2,cone-free}) are simply minimizing a s.c. function and have the stronger property of quadratic convergence for the points close enough to the central path \cite{interior-book}. Proving asymptotic quadratic convergence for a suitable variant of our algorithm is a future goal.   }
\end{remark}
\section{Analysis of the algorithms} \label{sec:analysis}
In this section, we analyze the predictor and corrector steps we defined in the previous section. This analysis lets us modify the framework for primal-dual algorithms  in Section \ref{sec:alg} to achieve the current best iteration complexity bounds. This modification and the main theorem about it come in Section \ref{ch:com-result}. 
The following lemma shows how to bound the proximity measure \eqref{DD-proximity-1-1-2} based on the local norm defined by the current primal and dual iterates:
\begin{lemma} \label{lem:prox-LN}
(a) Assume that $f(x)$ is an $a$-s.c.\ function and let $f_*(y)$ be its LF conjugate. Then, for every $x$ and $y$ in the domains of $f$ and $f_*$ we have
\begin{eqnarray}  \label{eq:lem:prox-LN}
\ \ \ \ a\rho \left ( r \right)  \leq f(x)+f_*(y) - \langle y , x \rangle  \leq a\rho \left (-  r \right),
\end{eqnarray}
where $r:=a^{-1/2}\|y-f'(x)\|_{[f''(x)]^{-1}}$ and $\rho$ is defined in \eqref{eq:rho}.\\
(b) Moreover, assume that there exist $\hat x$ and $\hat y$ in the domains of $f$ and $f_*$ respectively such that $\hat y = f'(\hat x)$ and $\langle x-\hat x ,y  -\hat y \rangle = 0$. Then, 
\begin{eqnarray}  \label{eq:lem:prox-LN-2}
a\rho(r)+a\rho(s)  \leq  f(x)+f_*(y) - \langle y , x \rangle  \leq  a\rho(-r)+a\rho(-s),
\end{eqnarray}
where $r:=a^{-1/2}\|x-\hat x\|_{f''( \hat x)}$ and $s:=a^{-1/2}\|y-\hat y\|_{f_*''(\hat y)}$. 
\end{lemma}
\begin{proof}
(a) By writing the second inequality in \eqref{property-4} for $f_*$ at two points $y$ and $f'(x)$, we have
\[
f_*(y) \leq f_*(f'(x)) + \langle f'_*(f'(x)), y-f'(x) \rangle +  a\rho(-a^{-1/2}\|y-f'(x)\|_{f_*''(f'(x))}).
\]
To get the RHS inequality in \eqref{eq:lem:prox-LN}, we substitute $f'_*(f'(x))=x$ and $f_*''(f'(x))=[f''(x)]^{-1}$ from \eqref{eq:LF-2}, and $f_*(f'(x))+f(x)=\langle f'(x),x\rangle$ from Theorem \ref{thm:FY}. 
The LHS inequality can be similarly proved by using the first inequality in \eqref{property-4}.    \\
(b) We write the property \eqref{property-4} for $f$ at $x$ and $\hat x$ and  for $f_*$ at $y$ and $\hat y$, and add them together. 
\end{proof}
\begin{corollary} \label{coro:bound-prox-1}
	For every $(x,\tau,y) \in Q_{DD}$, we have
\begin{eqnarray}  \label{eq:coro:bound-prox-1}
\ \ \ \rho \left(\left\|\frac{\tau y}{\mu}-  \Phi' \left (  u \right)\right\|_{[\Phi''(u)]^{-1}} \right) \leq \Omega_\mu(x,\tau,y) \leq \rho \left(-\left\|\frac{\tau y}{\mu}-  \Phi' \left (  u \right)\right\|_{[\Phi''(u)]^{-1}} \right), 
\end{eqnarray}
where $\mu:=\mu(x,\tau,y)$ and  $u:=Ax+\frac{1}{\tau} z^0$.
\end{corollary}
As we explained before, matrix $\bar H$ in \eqref{hessian-phi-2}  defines Dikin ellipsoid type properties that are crucial in our analysis. In both the predictor and corrector steps, we have inequality \eqref{eq:lem:ana-1-1} for a vector $d^\top:= [\bar d_x^\top \  \ d_\tau \ \ d_v^\top]$ as the solution of \eqref{eq:system-pred-1} for a proper scalar $q$. We can break down \eqref{eq:lem:ana-1-1} into several useful bounds for our analysis.
 First, clearly
\begin{eqnarray} \label{eq:dd-expand-2}
(\xi-1) \frac{d_\tau^2}{\tau^2}\vartheta  \leq q \ \ \Rightarrow \ \ \left (\frac{d_\tau}{\tau}  \right)^2 \leq  \frac{q}{(\xi-1) \vartheta}.
\end{eqnarray}
Using \eqref{eq:lem:ana-1-1} and \eqref{eq:dd-expand-2}, we get
\begin{eqnarray} \label{eq:dd-expand-3}
\frac{1}{\tau} \left \|A \bar d_x-d_\tau \left (Ax+\frac{1}{\tau}z^0 \right) \right \|_{\Phi''} \leq   \sqrt{q} + \left | \frac{d_\tau}{\tau}\right| \sqrt{\vartheta}
\underbrace{\leq}_{\text{\eqref{eq:dd-expand-2}}}  \left(1 +  \sqrt{\frac{1}{\xi-1}}\right) \sqrt{q}.
\end{eqnarray}
\eqref{eq:dd-expand-3} gives a bound on the displacement in $Ax+\frac{1}{\tau}z^0$ as shown in \eqref{eq:dd-expand-4}. 
 Also from \eqref{eq:lem:ana-1-1} we get
\begin{eqnarray} \label{eq:dd-expand-5}
\frac{\tau^2}{\mu^2} \langle d_y, [\Phi'']^{-1}d_y \rangle \leq  q. 
\end{eqnarray}
Let us see how to use these bounds in the analysis of the predictor and corrector steps. 
\subsection{Predictor step}
Let us first show how the predictor step increases $\mu$. For analyzing this, we prove a result about the structure of $U$ defined in \eqref{eq:mat-U}. We start with a lemma:
\begin{lemma}  \label{lem:proj-1}
Assume that $\cH$ is a symmetric positive definite matrix and $U$ is a matrix of proper size with linearly independent columns. Then, for any given vector $f$ of proper size, we have
\begin{eqnarray} \label{eq:corrector-5}
f^\top U \left( U^\top \cH U \right)^{-1} U^\top f = f^\top \cH^{-1}f- f^\top \cH^{-1}{U^\perp}^\top\left( U^\perp \cH^{-1} {U^\perp}^\top\right)^{-1} U^\perp \cH^{-1} f,
\end{eqnarray}
where $U^\perp$ is a matrix whose rows form a basis for the kernel of $U^\top$. 
\end{lemma}
\begin{proof}
As $\cH$ is symmetric positive definite and $U$ has linearly independent columns, the system $U^\top \cH U g=U^\top f$ has a unique solution $g$. By definition of $U^\perp$, there exists $w$ such that $\cH U g=f+{U^\perp}^\top w$. Multiplying both sides by $\cH^{-1}$ gives us $U g=\cH^{-1}f+\cH^{-1} {U^\perp}^\top w$. To calculate $w$, we multiply both sides of the last equation from the left by $U^\perp$. Note that $U^\perp U=0$ and $U^\perp \cH^{-1} {U^\perp}^\top$ is invertible. If we solve for $w$ and substitute it in $U g=\cH^{-1}f+\cH^{-1} {U^\perp}^\top w$, we get 
\begin{eqnarray} \label{eq:corrector-4}
Ug=\cH^{-1}f-\cH^{-1}{U^\perp}^\top\left( U^\perp \cH^{-1} {U^\perp}^\top\right)^{-1} U^\perp \cH^{-1} f.
\end{eqnarray}
If we multiply both sides of \eqref{eq:corrector-4} from the left by $f^\top$ and substitute $g=(U^\top \cH U)^{-1}U^\top f$, we get \eqref{eq:corrector-5}.
\end{proof}
We are interested in matrix $U \in \R^{(2m+2)\times (m+1)}$ defined in \eqref{eq:mat-U}, which has a very special structure. For this $U$, one option for $U^\perp$, defined in Lemma \ref{lem:proj-1}, is
\begin{eqnarray} \label{eq:mat-U-perp}
U^\perp=\left [\begin{array}{cccc}  
0  & c & A^\top & 0 \\ -c_A & 0 & 0 & 1\\ -F & 0 & 0 & 0  \end{array} \right]. 
\end{eqnarray} 
If we compare $U$ and $U^\perp$, we see that the rows of $U$ is a permutation of the columns of $U^\perp$. Explicitly
\begin{eqnarray} \label{eq:corrector-6}
U^\perp=U^\top P, \ \ \ \ \ P:=\left[ \begin{array}{cc} 0 & I_{m+1} \\ I_{m+1} & 0 \end{array}\right].
\end{eqnarray}
We have the following lemma:
\begin{lemma} \label{lem:proj-2}
Let $\bar H$ be a symmetric positive definite matrix and $\mu >0$. Assume the setup of Lemma \ref{lem:proj-1} where $\cH$ and $f$ have the form
\begin{eqnarray} \label{eq:corrector-7}
\cH:=\left[ \begin{array}{cc} \bar H & 0 \\ 0 & \frac{1}{\mu^2} \bar H^{-1} \end{array}\right], \ \ \ f:=\left[ \begin{array}{c} f_1 \\ f_2 \end{array}\right], 
\end{eqnarray}
such that $f_1$ and $f_2$ further satisfy $f_1=\mu \bar Hf_2$ or $f_1=-\mu \bar Hf_2$. Also assume that \eqref{eq:corrector-6} holds for $U$ and $U^\perp$. Then, 
\begin{eqnarray} \label{eq:corrector-9}
f^\top U \left( U^\top \cH U \right)^{-1} U^\top f = \frac 12 f^\top \cH^{-1}f.
\end{eqnarray}
\end{lemma}
\begin{proof}
We can verify that $\cH^{-1}=\mu^2P \cH P$ for $P$ defined in \eqref{eq:corrector-6}. Using this and \eqref{eq:corrector-6}, for the second term in the RHS of \eqref{eq:corrector-5} we have
\begin{eqnarray} \label{eq:corrector-8}
\ \ \ f^\top \cH^{-1}{U^\perp}^\top\left( U^\perp \cH^{-1} {U^\perp}^\top\right)^{-1} U^\perp \cH^{-1}f=\mu^2  \left[ \begin{array}{c} f_2 \\ f_1 \end{array}\right]^\top \cH U \left( U^\top \cH U \right)^{-1} U^\top \cH\left[ \begin{array}{c} f_2 \\ f_1 \end{array}\right].
\end{eqnarray}
Using $f_1=\mu \bar Hf_2$ or $f_1=-\mu \bar Hf_2$, \eqref{eq:corrector-8} equals $f^\top U \left( U^\top \cH U \right)^{-1} U^\top f$ and so \eqref{eq:corrector-5} reduces to \eqref{eq:corrector-9}. 
\end{proof}
Let us see how Lemma \ref{lem:proj-2} is useful for our setup. We define
\begin{eqnarray} \label{eq:corrector-14}
\psi^p:=\left[ \begin{array}{c}f_1 \\ f_2  \end{array}\right], \ \  f_1:= \left[ \begin{array}{c}\frac{1}{\tau} \Phi' \\ -\frac{1}{\tau} \langle \Phi',Ax+\frac{1}{\tau} z^0 \rangle - \frac{ \xi \vartheta}{\tau} \end{array}\right], \ \  f_2:= \left[ \begin{array}{c}   \frac{\tau}{\mu} \left(Ax+\frac{1}{\tau}z^0 \right)  \\ \frac{\tau}{\mu} \end{array}\right].
\end{eqnarray}
For matrix $\bar H$ defined in \eqref{hessian-phi-2}, we can directly verify 
\begin{eqnarray}   \label{eq:pred-4}
&& \frac{1}{\mu} \left[ \begin{array}{c}\frac{1}{\tau} \Phi' \\ -\frac{1}{\tau} \langle \Phi',Ax+\frac{1}{\tau} z^0 \rangle - \frac{ \xi \vartheta}{\tau} \end{array}\right] = -\bar H \left[ \begin{array}{c}   \frac{\tau}{\mu} \left(Ax+\frac{1}{\tau}z^0 \right)  \\ \frac{\tau}{\mu} \end{array}\right].
\end{eqnarray}
Therefore, $f_1=-\mu \bar Hf_2$ and so \eqref{eq:corrector-9} holds for our setup. 
Now, we prove the following lemma:
\begin{lemma}  \label{lem:corr-1}
Consider $\cH$ defined in \eqref{eq:system-pred-1} and $\psi^p$ defined in \eqref{eq:corrector-14} for a point $(x,\tau,y) \in Q_{DD}$. Then, we have
\begin{eqnarray} \label{eq:corrector-11}
\langle U^\top \psi^p, \left[U^\top  \cH\left(\bar H, \mu^2 \bar H\right) U \right]^{-1} U^\top \psi^p \rangle= \xi \vartheta.
\end{eqnarray}
\end{lemma}
\begin{proof}
Equation \eqref{eq:pred-4} confirms that $f_1=-\mu \bar Hf_2$, so we have equation \eqref{eq:corrector-9}. Hence, we need to show that $( \psi^p)^\top \cH^{-1} \psi^p=2\xi \vartheta$ to get our result.  This holds since by direct verification we have
\begin{eqnarray} \label{eq:pred-5}
 -\mu \left[ \begin{array}{c}\frac{1}{\tau} \Phi' \\ -\frac{1}{\tau} \langle \Phi',Ax+\frac{1}{\tau} z^0 \rangle - \frac{ \xi \vartheta}{\tau} \end{array}\right] ^\top \left[ \begin{array}{c}   \frac{\tau}{\mu} \left(Ax+\frac{1}{\tau}z^0 \right)  \\ \frac{\tau}{\mu} \end{array}\right] = \xi \vartheta, 
\end{eqnarray}
and $( \psi^p)^\top \cH^{-1} \psi^p$, by using \eqref{eq:pred-4},  is exactly the summation of two terms like \eqref{eq:pred-5}.
\end{proof}
Now we are ready to prove the following main proposition about how the predictor step increases $\mu$. 
\begin{proposition} \label{prop:mu-increase}
Assume that $(x,\tau,y) \in Q_{DD}$ and conditions \eqref{eq:lem:hessian-close-1-1} and \eqref{eq:lem:hessian-close-1-2} hold. Let our search direction be the solution of \eqref{eq:system-pred-1} with $r_{RHS}=r^0/\mu^2$ and any $\hat H$ that satisfies \eqref{eq:lem:hessian-close-1-2}. Let $\alpha_2 > 0$ be such that $\tau+\alpha_2 d_\tau>0$ and choose $\alpha_1=\frac{\alpha_2}{\tau+\alpha_2 d_\tau}$. Then, for the updates in \eqref{eq:update}  we have
\begin{eqnarray} \label{eq:prop:mu-increase-1}
(1-\bar \epsilon)^2 \alpha_2  \leq \mu(x^+, \tau^+, y^+)-\mu(x,\tau,y)  \leq  \frac{\alpha_2}{(1-\bar \epsilon)^2}.
\end{eqnarray}
\end{proposition}
\begin{proof}
A key to the proof is that on the central path we have $U^\top \psi^p(\mu)=-\frac{1}{\mu} r^0$, where $\psi^p$ is defined in \eqref{eq:corrector-14} and $r^0$ is defined in \eqref{eq:system-pred-1}. This can be directly verified by using \eqref{trans-dd-path-1-copy-2} and \eqref{eq:dd-4-2} for the points on the central path. 
By starting from \eqref{eq:pred-3} for  $\mu(x^+, \tau^+, y^+)-\mu(x,\tau,y)$, we can continue
\begin{eqnarray} \label{eq:dd-15-edited}
\begin{array}{rcl}
&& \mu(x^+, \tau^+, y^+)-\mu(x,\tau,y)  \\
&=& \frac{\alpha_2}{\xi \vartheta} [\bar d_x^\top \ \ \ d_\tau \ \ \ d_v^\top]  r^0  \\
&=&  \frac{\alpha_2}{\xi \vartheta}  \frac{1}{\mu^2}  \langle r^0, [U^\top \cH(\bar H, \hat H) U]^{-1} r^0 \rangle,     \hfill \text{using \eqref{eq:system-pred-1},}  \\
                              &=&  \frac{\alpha_2}{\xi \vartheta} \langle U^\top \psi^p(\mu), [U^\top  \cH(\bar H, \hat H) U ]^{-1} U^\top \psi^p(\mu) \rangle, \ \ \ \ \hfill \text{using $U^\top \psi^p(\mu)=-\frac{1}{\mu} r^0$}. 
\end{array}
\end{eqnarray}
We get the desired result by using conditions \eqref{eq:lem:hessian-close-1-1} and \eqref{eq:lem:hessian-close-1-2} and then utilizing Lemma \ref{lem:corr-1} for the points on the central path. 
\end{proof}
Proposition \ref{prop:mu-increase} implies that the amount of increase in $\mu$ depends directly on $\alpha_2$. Therefore, we need to show how large $\alpha_2$ can be chosen in the predictor step.

\begin{lemma}  \label{lem:ana-1}
Assume that $(x,\tau,y) \in Q_{DD}$  and conditions \eqref{eq:lem:hessian-close-1-1} and \eqref{eq:lem:hessian-close-1-2} hold. Then, \eqref{eq:lem:ana-1-1} holds with
$q:=\frac{1}{(1-\bar \epsilon)^6}  \frac{\xi \vartheta}{\mu^2}$ for the solution of \eqref{eq:system-pred-1} with $r_{RHS}=r^0/\mu^2$ and any $\hat H$ that satisfies \eqref{eq:lem:hessian-close-1-2}. 
\end{lemma}
\begin{proof}
Let us define $f=Ud$ for $d$ the solution of  \eqref{eq:system-pred-1}. Then, by using \eqref{eq:lem:hessian-close-1-1} and \eqref{eq:lem:hessian-close-1-2}, we have
\begin{eqnarray} \label{eq:lem:ana-1-3-2}
\begin{array}{rcl}
f^\top \cH(\bar H, \mu^2 \bar H) f &=& \frac{1}{\mu^4}\langle    \left[U^\top  \cH(\bar H, \hat H) U \right]^{-1} r^0 ,     (U^\top \cH(\bar H, \mu^2 \bar H) U)   \left[U^\top  \cH(\bar H, \hat H) U \right]^{-1}  r^0 \rangle  \hfill \\
&\leq& \frac{1}{(1-\bar \epsilon)^4 \mu^4}  \langle r^0, \left[U^\top  \cH(\bar H, \hat  H) U \right]^{-1} r^0 \rangle,  \hfill \text{using \eqref{eq:lem:hessian-close-1-1} and \eqref{eq:lem:hessian-close-1-2},}  \\
&\leq& \frac{1}{(1-\bar \epsilon)^6 \mu^4}  \langle r^0, \left[U^\top  \cH\left(\bar H(\mu), \mu^2 \bar H(\mu)\right) U \right]^{-1} r^0 \rangle, \hfill \text{using \eqref{eq:lem:hessian-close-1-1},}  \\
&=& \frac{\langle U^\top \psi^p(\mu), \left[U^\top  \cH\left(\bar H(\mu), \mu^2 \bar H(\mu)\right) U \right]^{-1} U^\top \psi^p(\mu) \rangle}{(1-\bar \epsilon)^6  \mu^2},   \hfill \text{using $U^\top \psi^p(\mu)=-\frac{1}{\mu} r^0$,}   \\
&=&\frac{1}{(1-\bar \epsilon)^6  \mu^2}   \xi \vartheta, \hfill \text{using Lemma \ref{lem:corr-1}}.
\end{array}
\end{eqnarray}
\end{proof}
We want to control the change in $\Omega_\mu(x,\tau,y)$ by using Corollary \ref{coro:bound-prox-1}.
In view of this, by adding and subtracting some terms, we have (with $\mu^+:=\mu(x^+,\tau^+,y^+)$)
\begin{eqnarray} \label{eq:dd-43}
\ \ \ \ \ \ \
\left(\frac{\tau^+ y^+}{\mu^+}-  \Phi' \left (  u^+ \right)\right) - \left(\frac{\tau y}{\mu}-  \Phi' \left (  u \right)\right)= \left (\frac{\tau^+}{\mu^+} - \frac{\tau }{\mu}\right) y + \frac{\tau^+}{\mu^+} \alpha_2 d_y - \left( \Phi' \left (  u^+ \right) - \Phi' \left (  u \right)\right).
\end{eqnarray}
Let us give a bound on the local norm defined by $\Phi''$ for the three terms in \eqref{eq:dd-43}. Using Proposition \ref{prop:mu-increase}, we have
\begin{eqnarray} \label{eq:dd-44}
\ \ \  \left |\frac{\tau^+}{\mu^+} - \frac{\tau }{\mu}\right|= \left |\frac{\tau+\alpha_2 d_\tau}{\mu^+} - \frac{\tau }{\mu}\right|= \left|\frac{\alpha_2 \mu d_\tau-\tau (\mu^+-\mu)}{\mu\mu^+} \right| \leq \alpha_2 \left ( \left | \frac{d_\tau}{\tau} \right| + \left|\frac{1}{\mu(1-\bar \epsilon)^2} \right|\right) \frac{\tau}{\mu}.
\end{eqnarray}
$\Omega_\mu(x,\tau,y) \leq \delta_1$ and \eqref{eq:coro:bound-prox-1} imply that $\left\|\frac{\tau y}{\mu}-  \Phi' \left (  u \right)\right\|_{[\Phi''(u)]^{-1}} \leq \sigma(\delta_1)$, where $\sigma$ is defined in \eqref{eq:sigma-1}. Then, by using \eqref{eq:LF-2} and property \eqref{property-3} for $\Phi_*$, assuming $\sigma(\delta_1) < 1$  we have
\begin{eqnarray} \label{eq:dd-45}
[\Phi''(u)]^{-1}= \Phi''_*(\Phi'(u)) \preceq \frac{1}{(1-\sigma(\delta_1))^2} \Phi_*''\left ( \frac{\tau y}{\mu}\right).
\end{eqnarray}
Using \eqref{eq:dd-44} and \eqref{eq:dd-45}, we can bound the local norm of the first term in the RHS of \eqref{eq:dd-43} as	
\begin{eqnarray} \label{eq:dd-46}
\begin{array}{rclr} 
 \left |\frac{\tau^+}{\mu^+} - \frac{\tau }{\mu}\right| \|y\|_{[\Phi''(u)]^{-1}} &\leq& \frac{\alpha_2}{1-\sigma(\delta_1)} \left ( \left | \frac{d_\tau}{\tau} \right| + \left|\frac{1}{\mu(1-\bar \epsilon)^2} \right|\right) \left\| \frac{\tau}{\mu}y \right\|_{\Phi_*''} & \\
 &\leq&  \left ( \left | \frac{d_\tau}{\tau} \right| + \left|\frac{1}{\mu(1-\bar \epsilon)^2} \right|\right)  \frac{\alpha_2}{1-\sigma(\delta_1)} \sqrt{\vartheta},  &  \text{using \eqref{eq:norm-phi*-1}}.
 \end{array}
\end{eqnarray}
For the second term in the RHS of \eqref{eq:dd-43} we have
\begin{eqnarray} \label{eq:dd-47}
\begin{array}{rclr}
\frac{\tau^+}{\mu^+} \alpha_2  \| d_y \|_{[\Phi''(u)]^{-1}} &\leq& \left[ 1+ \alpha_2 \left ( \left | \frac{d_\tau}{\tau} \right| + \left|\frac{1}{\mu(1-\bar \epsilon)^2} \right|\right)\right] \alpha_2 \left\| \frac{\tau}{\mu} d_y \right\|_{[\Phi''(u)]^{-1}},  & \text{using \eqref{eq:dd-44},}  \\  &\leq& \left[ 1+ \alpha_2 \left ( \left | \frac{d_\tau}{\tau} \right| + \left|\frac{1}{\mu(1-\bar \epsilon)^2} \right|\right)\right] \alpha_2 \sqrt{q}, & \text{using \eqref{eq:dd-expand-5}}.
\end{array}
\end{eqnarray}
For the third term, first by using \eqref{eq:dd-expand-4} and substituting the bound in \eqref{eq:dd-expand-3} we have
\begin{eqnarray} \label{eq:dd-48}
\left \|u^+-u \right \|_{\Phi''}\leq \underbrace{\frac{1}{1+\alpha_2 (d_\tau/\tau)}  \left(1 +  \sqrt{\frac{1}{\xi-1}}\right) \alpha_2 \sqrt{q}}_{=:\bar \delta}.
\end{eqnarray}
If we choose $\alpha_2$ such that $\bar \delta < 1$, then,  by Lemma \ref{lem:der-to-real}, we have
\begin{eqnarray} \label{eq:dd-49}
\left\| \Phi' \left (  u^+ \right) - \Phi' \left (  u \right)\right\|_{[\Phi''(u)]^{-1}}  \leq \frac{\bar \delta}{1-\bar \delta}.
\end{eqnarray}
Putting together the above bounds, we can prove the following main result:
\begin{proposition} \label{prop:dd-4}
Assume that  $0.2 > \delta_2 > 4 \delta_1 > 0$ and for a point $(x,\tau,y) \in Q_{DD}$ we have $\Omega_\mu(x,\tau,y) \leq \delta_1$. Let the predictor step be calculated from \eqref{eq:system-pred-1} with $r_{RHS}=r^0/\mu^2$ and any $\hat H$ that satisfies \eqref{eq:lem:hessian-close-1-2}. Then, there exists a positive constant $\kappa_1$ depending on $\delta_1$, $\delta_2$, and $\xi$  such that we can choose $\alpha_2$ large enough to satisfy 
\begin{eqnarray} \label{eq:dd-16}
\alpha_2 \geq \frac{\kappa_1}{\sqrt{ \vartheta}} \mu,
\end{eqnarray}
and $\alpha_1:=\frac{\alpha_2}{\tau+\alpha_2 d_\tau}$ for the update of \eqref{eq:update} while $\Omega_\mu(x^+,\tau^+,y^+) \leq \delta_2$.
\end{proposition}
\begin{proof}
We choose $\alpha_2$ to make sure that  $\bar \delta$ defined in \eqref{eq:dd-48} satisfies $\bar \delta \leq 1/4$. To achieve this, we first assume that $\alpha_2 |d_\tau/\tau| \leq 1/2$, and then in view of \eqref{eq:dd-48} we choose $2(1+1/\sqrt{\xi-1})\alpha_2 \sqrt{q} \leq 1/4$. If we substitute the value of $q=\frac{1}{(1-\bar \epsilon)^6}  \frac{\xi \vartheta}{\mu^2}$ defined in Lemma \ref{lem:ana-1} and also use the bound in \eqref{eq:dd-expand-2}, the following inequality guarantees $\bar \delta \leq 1/4$:
\begin{eqnarray} \label{eq:extra-1}
\frac{\alpha_2 \sqrt{\vartheta}}{\mu} \leq \underbrace{\min \left\{\sqrt{\frac{\xi-1}{\xi}}\frac{(1-\bar \epsilon)^3}{2}, \frac{1}{8(1+1/\sqrt{\xi-1})\sqrt{\xi}} (1-\bar \epsilon)^3 \right\}}_{=:\kappa_{1,1}}.
\end{eqnarray}
Consider the bound we have for the proximity measure in Corollary \ref{coro:bound-prox-1}. Assuming that $\bar \delta$ defined in \eqref{eq:dd-48} satisfies $\bar \delta \leq 1/4$, by using property \eqref{property-3}, we have 
\begin{eqnarray} \label{eq:dd-50}
\begin{array}{rcl}
\left \| \frac{\tau^+ y^+}{\mu^+}-  \Phi' \left (  u^+ \right)  \right \|^*_{\Phi''(u^+)} 
&\leq &\frac{4}{3} \left \| \frac{\tau^+ y^+}{\mu^+}-  \Phi' \left (  u^+ \right)  \right \|^*_{\Phi''(u)} \\
&\leq&\frac{4}{3} \left \| \frac{\tau^+ y^+}{\mu^+}-  \Phi' \left (  u^+ \right) - \frac{\tau y}{\mu}+  \Phi' \left (  u \right) \right \|^*_{\Phi''(u)} +\frac{4}{3} \left \| \frac{\tau y}{\mu}-  \Phi' \left (  u \right) \right \|^*_{\Phi''(u)}   \\
&\leq&  \frac{4}{3}\left \| \frac{\tau^+ y^+}{\mu^+}-  \Phi' \left (  u^+ \right) - \frac{\tau y}{\mu}+  \Phi' \left (  u \right) \right \|^*_{\Phi''(u)}  + \frac{4}{3}\sigma (\delta_1),
\end{array}
\end{eqnarray}
where $\sigma(\cdot)$ is the inverse of $\rho(\cdot)$ defined in \eqref{eq:sigma-1}. Similarly, we define the inverse of $\rho(-\cdot)$ as $\bar \sigma(\cdot)$. To satisfy $\Omega_\mu(x^+,\tau^+,y^+) \leq \delta_2$, in view of Corollary \ref{coro:bound-prox-1} and using \eqref{eq:dd-50}, a sufficient condition is
\begin{eqnarray} \label{eq:dd-50-2}
\left \| \frac{\tau^+ y^+}{\mu^+}-  \Phi' \left (  u^+ \right) - \frac{\tau y}{\mu}+  \Phi' \left (  u \right) \right \|_{[\Phi''(u)]^{-1}}  \leq  \frac{3}{4}\bar \sigma (\delta_2)-\sigma (\delta_1).
\end{eqnarray}
For this analysis, we need to choose $\delta_1$ and $\delta_2$ such that $\frac{3}{4} \bar \sigma(\delta_2) >  \sigma(\delta_1)$. To force this, we choose $0.2 > \delta_2 > 4 \delta_1$; we can check that $\delta_2\geq \rho(-\frac{4}{3}(\sqrt{\delta_2/2}+\delta_2/4))$ for $\delta_2 \in (0,0.2)$, then we apply $\bar \sigma$ to both sides and use $\sigma(\delta_1)\leq \sqrt{2\delta_1}+\delta_1$ by \cite{cone-free}-Lemma 2.1. We have split the term inside the norm in the LHS of \eqref{eq:dd-50-2} into three terms in \eqref{eq:dd-43} and bounded the local norm for each of them. 
We add the bounds in \eqref{eq:dd-46}, \eqref{eq:dd-47}, and \eqref{eq:dd-49}. Then, by substituting $q=\frac{1}{(1-\bar \epsilon)^6}  \frac{\xi \vartheta}{\mu^2}$ and the bound in  \eqref{eq:dd-expand-2}, and considering $\bar \delta \leq 1/4$ and $\alpha_2 |d_\tau/\tau| \leq 1/2$, we can bound the LHS of \eqref{eq:dd-50-2} from above by 
\begin{eqnarray} \label{eq:dd-50-3}
\begin{array}{rcl}
\underbrace{\left[ \frac{\left ( \sqrt{\frac{\xi}{\xi-1}} \frac{1}{(1-\bar \epsilon)^3}+ \frac{1}{(1-\bar \epsilon)^2} \right)}{1-\sigma(\delta_1)} + 2 \frac{\sqrt{\xi}}{(1-\bar \epsilon)^3 }+\frac{8}{3} \left(1+\sqrt{\frac{1}{\xi-1}} \right) \frac{\sqrt{\xi }}{(1-\bar \epsilon)^3}\right]}_{=:1/\kappa_{1,2}}\frac{\alpha_2 \sqrt{\vartheta}}{\mu}. 
\end{array}
\end{eqnarray}
Note that for \eqref{eq:dd-47}, the term inside the bracket is bounded from above by 2  using the fact that we force \eqref{eq:dd-46} to be smaller than 1. Therefore, if we choose
\[
\frac{\alpha_2 \sqrt{\vartheta}}{\mu}=\kappa_1:=\min \left\{ \kappa_{1,1}, \kappa_{1,2} \left(\frac{3}{4}\bar \sigma (\delta_2)-\sigma (\delta_1)\right)\right\},
\]
then $\Omega_\mu(x^+,\tau^+,y^+) \leq \delta_2$ holds, which concludes the proof.  
\end{proof}
To complete the whole discussion, we need to prove Lemma \ref{lem:hessian-close-1}. Let us start with the following lemma: 
\begin{lemma} \label{lem:dd-8}
 For every set of points $(z,\tau,y,y_\tau,\mu)$ such that $u:=\frac{z}{\tau} \in D$, $y\in D_*$, $\mu >0$, and $y_\tau+\frac{1}{\tau} \langle y,z \rangle +\frac{\mu \xi \vartheta}{\tau}=0$, we have
\begin{eqnarray}  \label{eq:lem:dd-8}
\begin{array}{c}
	\left \|\frac{\tau y}{\mu}-  \Phi' \left (  u \right) \right\|_{[\Phi''(u)]^{-1}}   \leq \beta \leq \sqrt{\frac{\xi}{\xi-1}} \left \|\frac{\tau y}{\mu}-  \Phi' \left (  u \right) \right\|_{[\Phi''(u)]^{-1}},
\end{array}
\end{eqnarray}
where
\begin{eqnarray}  \label{eq:dd-17-2}
\begin{array}{c}
\beta (z,\tau,y,y_\tau,\mu) : =    \left \| \frac{1}{\mu} \left[ \begin{array}{c} y \\ y_\tau \end{array}\right] -  \left[\begin{array}{c} \frac{1}{\tau} \Phi'(u) \\ -\frac{1}{\tau} \langle \Phi'(u),u \rangle - \frac{ \xi \vartheta}{\tau}\end{array}  \right]  \right \|_{[\bar H(u,\tau)]^{-1}},
\end{array}
\end{eqnarray}
for $\bar H(u,\tau)$ defined in \eqref{hessian-phi-2} as a function of $u$ and $\tau$. 
\end{lemma}
\begin{proof}
Consider the definition of $\bar H$ in \eqref{hessian-phi-2} and the formula for its inverse in \eqref{inverse-formula-2-2}.  
We want to substitute $w:=\frac{y}{\mu}- \frac{1}{\tau} \Phi' \left (u \right)$ and $w_\tau:=\frac{y_\tau}{\mu} +\frac{1}{\tau} \langle \Phi'(u),u\rangle + \frac{  \xi \vartheta}{\tau}$ in \eqref{inverse-formula-2-2}. Note that by using the hypothesis of the lemma, we have
\begin{eqnarray*}
w_\tau=\frac{y_\tau}{\mu} +\frac{1}{\tau} \langle \Phi'(u),u\rangle + \frac{  \xi \vartheta}{\tau}=-\langle \frac{y}{\mu}-\frac{1}{\tau}\Phi', \frac{z}{\tau} \rangle. 
\end{eqnarray*}
Hence, by substituting this formula for $w_\tau$ and also $w$ in \eqref{inverse-formula-2-2}, we get
\begin{eqnarray}  \label{eq:dd-18}
\begin{array}{rcl}
\beta^2 &=&  \left\|\frac{\tau y}{\mu}- \Phi' \left (  u \right)\right\|^2_{[\Phi''(u)]^{-1}}+  \frac{\left[ \langle \frac{\tau y}{\mu} - \Phi', [\Phi'']^{-1} \Phi' \rangle \right] ^2}{\xi \vartheta  -\langle \Phi', [\Phi'']^{-1} \Phi' \rangle}   \\
	&\leq & \left\|\frac{\tau y}{\mu}- \Phi' \left (  u \right)\right\|^2_{[\Phi''(u)]^{-1}} + \frac{\left\|\frac{\tau y}{\mu}- \Phi' \left (  u \right)\right\|^2_{[\Phi''(u)]^{-1}}  \vartheta}{(\xi-1) \vartheta} 
	 =  \frac{\xi}{\xi-1} \left\|\frac{\tau y}{\mu}-  \Phi' \left (  u \right)\right\|^2_{[\Phi''(u)]^{-1}},
	 \end{array}
\end{eqnarray}
	where for the inequality we used CS inequality and property \eqref{property-2-2} of $\vartheta$-s.c.\ barriers. \eqref{eq:dd-18} immediately gives us \eqref{eq:lem:dd-8}. 
\end{proof}

\begin{proof}[Proof of Lemma \ref{lem:hessian-close-1}]
Assume that $\Omega_\mu(x,\tau,y) \leq \epsilon < 1$, by Corollary \ref{coro:bound-prox-1}, we have
	\[
	\rho \left (\left \|\frac{\tau y}{\mu}-  \Phi' \left (  u \right) \right\|_{[\Phi''(u)]^{-1}} \right)  \leq \epsilon  \ \Rightarrow \ \left \|\frac{\tau y}{\mu}-  \Phi' \left (  u \right) \right\|_{[\Phi''(u)]^{-1}}  \leq \sigma (\epsilon), 
	\]
where $\sigma(\cdot)$, defined in \eqref{eq:sigma-1}, is the inverse of $\rho(\cdot)$ for nonnegative values. 
If we define $(z,\tau):=(\tau Ax+z^0,\tau)$ and $\frac{1}{\mu} (y,y_\tau):=\frac{1}{\mu} (y,y_{\tau,0}+\tau \langle c,x \rangle)$, the hypotheses of Lemma \ref{lem:dd-8} are satisfied. Then, we have $\beta \leq \sqrt{\frac{\xi}{\xi-1}} \sigma(\epsilon)$. In Remark \ref{rem-1}, we mentioned that $\bar H(x,\tau)$, with some change of variables, is the Hessian of $f:=\Phi \left (\frac{z}{\tau} \right)-\xi \vartheta \ln(\tau)$, which we proved in Lemma \ref{lem:dd-9} that is a $\bar \xi$-s.c.\ function for an absolute constant $\bar \xi$ depending on $\xi$. 
We want to use Lemma \ref{lem:prox-LN} for $f$ and its conjugate at the points $(z,\tau)$ and $\frac{1}{\mu} (y,y_\tau)$, and the corresponding points with the same $\mu$ on the central path. 
One can verify that condition of Lemma \ref{lem:prox-LN}-(b) holds for these points, i.e., 
\begin{eqnarray}  \label{eq:lem:hessian-close-1-3}
&& \langle y - y(\mu) , z-z(\mu)  \rangle +  (y_\tau-y_\tau(\mu)) (\tau-\tau(\mu)) \nonumber \\
&=& \langle A^\top (y - y(\mu)), \tau x - \tau(\mu) x(\mu) \rangle + (\tau-\tau(\mu))  \langle c, \tau x - \tau(\mu) x(\mu) \rangle  \nonumber \\
&=&-(\tau-\tau(\mu))  \langle c, \tau x - \tau(\mu) x(\mu) \rangle+(\tau-\tau(\mu))  \langle c, \tau x - \tau(\mu) x(\mu) \rangle 
= 0. 
\end{eqnarray}
Note that the terms in the middle of both parts (a) and (b) of Lemma \ref{lem:prox-LN} are the same. If we use the upper bound from \eqref{eq:lem:prox-LN} and the lower bound from \eqref{eq:lem:prox-LN-2} and ignore one term in the LHS, we get
\begin{eqnarray} \label{eq:lem:hessian-close-1-4}
\begin{array}{rcl}
&& \rho \left (\frac{1}{\sqrt{\bar \xi}} \left(\bar H(x(\mu),\tau(\mu)) [z-z(\mu), \tau-\tau(\mu)]\right)^{1/2}  \right) \leq   \rho \left (-\frac{\beta}{\sqrt{\bar \xi}} \right)  \\
&\Rightarrow&  \left(\bar H(x(\mu),\tau(\mu)) [z-z(\mu), \tau-\tau(\mu)]\right)^{1/2}  \leq  \sqrt{\bar \xi}  \sigma \left( \rho \left (-\frac{1}{\sqrt{\bar \xi}} \sqrt{\frac{\xi}{\xi-1}} \sigma(\epsilon) \right)  \right). 
\end{array}
\end{eqnarray}
We have $\sigma(\epsilon) \leq \sqrt{2\epsilon}+\epsilon$ by \cite{cone-free}-Lemma 2.1, and for $\epsilon \leq 0.1$ we can easily verify that $\sqrt{2\epsilon}+\epsilon \leq \sqrt{3\epsilon}$. Also we can verify that for $t \leq 0.6$, we have $\rho(-t) \leq t^2$. Assume that $\sigma(\epsilon)$ is small enough to have $\sqrt{\frac{\xi}{\bar \xi(\xi-1)}} \sigma(\epsilon)  \leq 0.6$. Then, the RHS of \eqref{eq:lem:hessian-close-1-4} becomes
\begin{eqnarray} \label{eq:lem:hessian-close-1-4-2}
\leq \sqrt{\bar \xi}  \sigma \left( \frac{\xi}{\bar \xi(\xi-1)}\sigma^2(\epsilon)  \right)  \leq 3  \sqrt{\frac{\xi}{\xi-1}}  \sqrt{\epsilon}. 
\end{eqnarray}
Now we just need to use property  \eqref{property-3} of s.c.\ functions for $f=\Phi \left (\frac{z}{\tau} \right)-\xi \vartheta \ln(\tau)$ to get the result of the lemma.
\end{proof}
Before analyzing the corrector step, let us elaborate more on the above proof. 
For a point $(x,\tau,y) \in Q_{DD}$ with parameter $\mu$, let us define
\begin{eqnarray} \label{eq:lem:corr-ana-2}
d:= \left[ \begin{array}{c}   \tau(\mu) x(\mu)-\tau x  \\  \tau(\mu) -\tau \\  v(\mu)-v  \end{array} \right].
\end{eqnarray}
We can easily verify that  (using $y=y^0-(\tau-1) c_A - F^\top v$):
\begin{eqnarray} \label{eq:lem:corr-ana-3}
U d = \left [ \begin{array}{c}  \tau(\mu) Ax(\mu)+z^0 \\ \tau(\mu) \\ y(\mu) \\ y_{\tau,0}+\tau(\mu) \langle c,x(\mu) \rangle \end{array} \right]-\left [ \begin{array}{c}  \tau Ax+z^0 \\ \tau \\ y \\ y_{\tau,0}+\tau \langle c,x \rangle \end{array} \right].
\end{eqnarray}
We want to use property \eqref{property-3} for $r=1/4$ to change the local norm in \eqref{eq:lem:hessian-close-1-4}; it suffices to force $3  \sqrt{\frac{\xi}{\bar \xi(\xi-1)}} \sqrt{\epsilon} \leq \frac{1}{4}$ in view of \eqref{eq:lem:hessian-close-1-4-2}. Consider the proof of Lemma \ref{lem:hessian-close-1} and also the term for $y$ that we ignored in \eqref{eq:lem:hessian-close-1-4}. Then, using \eqref{eq:lem:hessian-close-1-4-2} and the above discussion, we have
\begin{corollary} \label{cor:bound-1} 
If for a point $(x,\tau,y)\in Q_{DD}$ we have  $3  \sqrt{\frac{\xi}{\bar \xi(\xi-1)}} \sqrt{\Omega_\mu(x,\tau,y)} \leq \frac{1}{4}$, then for $d$ defined in \eqref{eq:lem:corr-ana-2} we have
\begin{eqnarray}  \label{eq:corr-proof-1}
\| d\|_{U^\top \cH(\bar H(x,\tau), \mu^2 \bar H(x,\tau)) U} \leq 2\cdot \frac{4}{3} \left(3 \sqrt{\frac{\xi}{\xi-1}} \sqrt{\Omega_\mu(x,\tau,y)}\right) = \underbrace{8  \sqrt{\frac{\xi}{\xi-1}}}_{=:\bar \xi_1}  \sqrt{\Omega_\mu(x,\tau,y)}.
\end{eqnarray}
\end{corollary}
This inequality gives us \eqref{eq:lem:ana-1-1} for $q=\bar \xi_1^2 \Omega_\mu(x,\tau,y)$ that we break down to get the bounds we need for the analysis of the corrector step.  
\subsection{Corrector step}
We focus on the case that $\alpha_1=\frac{\alpha_2}{\tau+\alpha_2 d_\tau}$ (assuming $\tau+\alpha_2d_\tau > 0$) in the updates of \eqref{eq:update}. By Remark \ref{rem-2}, $\mu^+=\mu$ for every $\alpha_2$ and so we just need to show that $\alpha_2$ can be chosen to get enough reduction in the proximity measure. 
Let $d^c$ be the corrector step derived by solving \eqref{eq:system-pred-1} with parameters defined in \eqref{eq:system-corrector-step-1}. We argued
by using \eqref{eq:corrector-16-2} that the value of $\langle d^c, U^\top \psi ^c \rangle $ represents the first order reduction in $\Omega_\mu$.   
On the other hand, by using \eqref{eq:system-pred-1} and \eqref{eq:system-corrector-step-1}, we can verify
\begin{eqnarray} \label{eq:corrector-36}
-\langle d^c,  U^\top \psi ^c \rangle = \left \|U^\top \psi ^c+ \beta r^0 \right \|^2_{(U^\top \cH U)^{-1}}.
\end{eqnarray}
The following key lemma shows that this quantity has a large enough value:
\begin{lemma} \label{lem:corr-ana}
Let $(x,\tau,y) \in Q_{DD}$. If 
\begin{eqnarray} \label{eq:bound-omega}
\Omega_\mu(x,\tau,y) \leq \frac{1}{100 \left((\bar \xi_2 \bar \xi_1 )^3+\bar \xi_3 \bar \xi_1^3\right)^2 },
\end{eqnarray}
where $\bar \xi_1=8\sqrt{\xi}/\sqrt{\xi-1}$ is defined in \eqref{eq:corr-proof-1} and 
\begin{eqnarray} \label{eq:xi-1}
\begin{array}{ll}
\bar \xi_2 := 3\sqrt{\frac{1}{\xi-1}} +\frac{7}{2},  & \ \ \ 
\bar \xi_3 := \frac{1}{2\sqrt{\xi-1}} \left(\frac{11}{2}+\frac{5}{\sqrt{\xi-1}}\right) \left(3+\frac{2}{\sqrt{\xi-1}} \right)+\frac{2}{\xi-1}\left(1+\frac{1}{\sqrt{\xi-1}} \right),
\end{array}
\end{eqnarray}
then,
\begin{eqnarray*}
\left\|U^\top \psi ^c+ \beta r^0\right\|_{(U^\top \cH U)^{-1}}   \geq   \frac{1}{4\bar \xi_1}  \sqrt{\Omega_\mu(x,\tau,y)},
\end{eqnarray*}
where $\beta$ is defined in \eqref{eq:system-corrector-step-1} and $\psi^c$ is defined in \eqref{eq:corrector-10-2}.
\end{lemma}
\begin{proof}
Note that $-\|U^\top \psi ^c+ \beta r^0\|_{(U^\top \cH U)^{-1}}$ is the optimal objective value of \eqref{eq:lem:corr-ana-1} and
we find an upper bound for it by using a specific feasible solution. Our feasible solution is
\begin{eqnarray} \label{eq:lem:corr-ana-2-2}
\frac{ d}{\| d\|_{U^\top \cH(\bar H, \mu^2 \bar H) U}},
\end{eqnarray}
where $d$ is defined in \eqref{eq:lem:corr-ana-2} and we have Corollary \ref{cor:bound-1} for a bound on its local norm. We can verify that \eqref{eq:lem:corr-ana-2-2} satisfies all the constraints. 
 Now, we need to prove that $-\langle  d, U^\top \psi ^c \rangle$ is large enough. The idea of the proof is that we consider the bounds in \eqref{eq:corrector-16-2} at $\alpha_2=1$ and $\alpha_2=2$, and if $-\langle  d, U^\top \psi ^c \rangle$ is not large enough, we get a contradiction.

For simplicity, let $\Omega_\mu:=\Omega_\mu(x,\tau,y)$ and define $\sqrt{q}=\bar \xi_1 \sqrt{\Omega_\mu}$ for $\bar \xi_1$ defined in \eqref{eq:corr-proof-1}. Then \eqref{eq:lem:ana-1-1} becomes the expansion of \eqref{eq:corr-proof-1} and we have all the inequalities we extracted after Lemma \ref{lem:ana-1}, which we use to find bounds for the terms we have in \eqref{eq:corrector-16-2}. For the first term of $D(\alpha_2)$ we can use \eqref{eq:dd-expand-3}. 
For the second term of $D(\alpha_2)$ we use triangle inequality and we have
\begin{eqnarray} \label{eq:lem:corr-ana-6}
\left \| \frac{d_\tau y}{\mu} \right\|_{\Phi_*''} =\frac{d_\tau}{\tau}\left \| \frac{\tau y}{\mu} \right\|_{\Phi_*''}  \leq  \sqrt{\frac{q}{\xi-1}},  \ \ \ \text{using \eqref{eq:dd-expand-2} and \eqref{eq:norm-phi*-1}},
\end{eqnarray}
and using \eqref{eq:dd-expand-2} and \eqref{eq:dd-expand-5}, we have
\begin{eqnarray} \label{eq:lem:corr-ana-7}
\left \| \frac{ (\tau +\alpha_2 d_\tau) d_y}{\mu} \right\|_{\Phi_*''} &=&  \frac{(\tau +\alpha_2 d_\tau)}{\tau}\left \| \frac{ \tau d_y}{\mu} \right\|_{\Phi_*''} 
 \leq \left(1+\alpha_2 \sqrt{\frac{q}{(\xi-1)\vartheta}}\right) \sqrt{q}.
\end{eqnarray} 
If we use the CS inequality \eqref{eq:CS} for $B=\Phi''$ and use $\|\Phi'\|_{[\Phi'']^{-1}} \leq \sqrt{\vartheta}$ (see \eqref{property-2-2}), then \eqref{eq:dd-expand-2} and \eqref{eq:dd-expand-3} imply
\begin{eqnarray}  \label{eq:corrector-37}
\left |\frac{ d_\tau}{\tau(\tau+ \alpha_2 d_\tau)}  \langle \Phi', A \bar d_x -d_\tau \left(Ax+\frac 1\tau z^0 \right) \rangle \right|  
\leq \frac{\sqrt{\frac{q}{\xi-1}}}{1-\alpha_2 \sqrt{\frac{q}{(\xi-1)\vartheta}}}\left(1+\sqrt{\frac{1}{\xi-1}} \right)\sqrt{q}.
\end{eqnarray}
We want to make the second line of the term in the middle of inequalities in \eqref{eq:corrector-16-2} a quadratic in terms of $\alpha_2$, while the upper and lower bounds are proportional to $\alpha_2^3$. To do this, we modify \eqref{eq:corrector-16-2} by adding and subtracting some terms to all sides as:
\begin{eqnarray}\label{eq:corrector-16-3}
&&\begin{array}{rcl}
&& \rho(D(\alpha_2))-\frac{1}{2}(D(\alpha_2))^2+\hat D(\alpha_2)   \\
& \leq &\Omega_\mu(x^+ ,\tau^+,y^+)-\Omega_\mu(x,\tau,y) -\alpha_2   \left [ \bar d_x^\top \ \ d_\tau \  \  d^\top_v \right]  U^\top  \psi^c  \\
&&+ \frac{\alpha_2^2 d_\tau}{\tau^2}  \langle \Phi', A \bar d_x -d_\tau \left(Ax+\frac 1\tau z^0 \right) \rangle - \frac{\alpha_2^2 d_\tau}{\mu}  \left ( \langle d_y,\Phi_*' \rangle + \langle c, \bar d_x \rangle \right) -\frac{1}{2}(\bar D(\alpha_2))^2  \\
&\leq &\rho \left(-D(\alpha_2) \right)-\frac{1}{2}(D(\alpha_2))^2+ \hat D(\alpha_2),\\
 \bar D(\alpha_2)&:=& \frac{\alpha_2}{\tau}\left \| A \bar d_x -d_\tau \left(Ax+\frac 1\tau z^0 \right) \right\|_{\Phi''}+\alpha_2\left \| \frac{d_\tau y + \tau  d_y}{\mu} \right\|_{\Phi_*''},  \\
 \hat D(\alpha_2) &:=& \frac{1}{2} \left((D(\alpha_2))^2-(\bar D(\alpha_2))^2\right)+\frac{\alpha_3^2 d^2_\tau}{\tau^2(\tau+ \alpha_2 d_\tau)}  \langle \Phi', A \bar d_x -d_\tau \left(Ax+\frac 1\tau z^0 \right) \rangle.
\end{array} 
\end{eqnarray}
 Note that by definition \eqref{eq:rho}, we can verify that
\begin{eqnarray} \label{eq:lem:corr-ana-8}
\rho(-t)-\frac{t^2}{2}  \leq t^3,  \ \ \ \ \frac{t^2}{2}-\rho(t) \leq t^3,  \ \ \ \forall t \in (0,0.8).
\end{eqnarray} 
Let us assume that $2\sqrt{\frac{q}{(\xi-1)\vartheta}} \leq \frac 12$, then \eqref{eq:dd-expand-3}, \eqref{eq:lem:corr-ana-6}, \eqref{eq:lem:corr-ana-7}, and \eqref{eq:corrector-37} yield that for $\alpha_2 \in (0,2)$ we have
\begin{eqnarray} \label{eq:lem:corr-ana-D1}
|D(\alpha_2)| &\leq& \alpha_2 \bar \xi_2 \sqrt{q} = \alpha_2 \bar \xi_2 \bar \xi_1 \sqrt{\Omega_\mu}, \nonumber \\
|\bar D(\alpha_2)| &\leq& \alpha_2 \left(2\sqrt{\frac{1}{\xi-1}} +2\right) \bar \xi_1 \sqrt{\Omega_\mu}, \nonumber \\
 |\hat D(\alpha_2)| &\leq& \alpha_2^3 \bar \xi_3  \bar \xi_1^3\Omega_\mu^{3/2},
\end{eqnarray}
where $\bar \xi_2$ nd $\bar \xi_3$ are defined in \eqref{eq:xi-1}. 
For the bound on $|\hat D(\alpha_2)|$, we also used the fact that 
\begin{eqnarray*} \label{eq:lem:corr-ana-D1-2}
|D(\alpha_2)-\bar D(\alpha_2)| \leq \frac{\alpha_2^2 |d_\tau|}{(\tau+\alpha_2 d_\tau)\tau} \left \| A \bar d_x -d_\tau \left(Ax+\frac 1\tau z^0 \right) \right\|_{\Phi''}+\alpha^2_2\left \| \frac{d_\tau  d_y}{\mu} \right\|_{\Phi_*''}.
\end{eqnarray*}
If we have $\bar \xi_2 \bar \xi_1 \sqrt{\Omega_\mu}  \leq 0.8$, by using \eqref{eq:lem:corr-ana-8} and \eqref{eq:lem:corr-ana-D1}, the middle term of \eqref{eq:corrector-16-3} is squeezed between $\pm \alpha_2^3 \left((\bar \xi_2 \bar \xi_1 )^3+\bar \xi_3 \bar \xi_1^3\right) \Omega_\mu^{3/2}$ for $\alpha_2 \in (0,2)$. We want to choose $\Omega_\mu$ small enough to make the term in the middle of \eqref{eq:corrector-16-3} be squeezed between  $\pm \frac{1}{10} \Omega_\mu$ for $\alpha_2=1$;  it suffices to have
\begin{eqnarray} \label{eq:lem:corr-ana-D2}
\left((\bar \xi_2 \bar \xi_1 )^3+\bar \xi_3 \bar \xi_1^3\right)  \Omega_\mu^{3/2}  \leq \frac{1}{10}  \Omega_\mu  \ \ \underbrace{\Leftrightarrow}_{\text{for $\Omega_\mu>0$}} \ \ \Omega_\mu \leq \frac{1}{100 \left((\bar \xi_2 \bar \xi_1 )^3+\bar \xi_3 \bar \xi_1^3\right)^2 }.
\end{eqnarray}
We claim that in this case, $-\langle  d, U^\top \psi ^c \rangle \geq \frac{1}{4} \Omega_\mu$. If we substitute $\alpha_2=1$, then $\Omega_\mu(x^+ ,\tau^+,y^+)=0$ as we can verify that  the point lays on the central path. Suppose for the sake of reaching a contradiction $-\langle d, U^\top \psi ^c \rangle < \frac{1}{4} \Omega_\mu$. Then, in view of \eqref{eq:corrector-16-3}, we must have
\begin{eqnarray*} \label{eq:lem:corr-ana-4}
 \frac{ d_\tau}{\tau^2}  \langle \Phi', A \bar d_x -d_\tau \left(Ax+\frac 1\tau z^0 \right) \rangle - \frac{ d_\tau}{\mu}  \left ( \langle \Phi_*', d_y \rangle + \langle c, \bar d_x \rangle \right) - \frac{1}{2}(\bar D(1))^2  \geq \left(\frac{3}{4}-\frac{1}{10} \right)  \Omega_\mu. 
\end{eqnarray*}
We reach our contradiction when we consider $\alpha_2=2$. 
For $\alpha_2=2$ we have $\Omega_\mu(x^+ ,\tau^+,y^+) \geq 0$. The term in the second line of \eqref{eq:corrector-16-2} is degree 2 of $\alpha_2$ and so becomes at least $\left(\frac{12}{4}-\frac{4}{10} \right) \Omega_\mu$ for $\alpha_2=2$. Then, at $\alpha_2=2$, \eqref{eq:corrector-16-3} implies
\begin{eqnarray*} \label{eq:lem:corr-ana-9}
-\Omega_\mu(x,\tau,y)+ \left(\frac{12}{4}-\frac{4}{10} \right) \Omega_\mu(x,\tau,y)  \leq \frac{8}{10} \Omega_\mu(x,\tau,y),
\end{eqnarray*} 
which is a contradiction. 

Now, if we consider the feasible solution  \eqref{eq:lem:corr-ana-2-2}  for the optimization problem \eqref{eq:lem:corr-ana-1} and putting together the bounds $-\langle  d, U^\top \psi ^c \rangle \geq \frac{1}{4} \Omega_\mu$ and $\| d\|_{U^\top \cH(\bar H, \mu^2 \bar H) U} \leq \bar \xi_1 \sqrt{\Omega_\mu}$ from \eqref{eq:corr-proof-1}, we get the result of the lemma. We can verify that for $\xi>1$, \eqref{eq:bound-omega} implies the other bounds we used for $\Omega_\mu$ in the proof, including the hypothesis of Corollary \ref{cor:bound-1}, $\bar \xi_2 \bar \xi_1 \sqrt{\Omega_\mu}  \leq 0.8$, and $2\sqrt{\frac{q}{(\xi-1)\vartheta}} \leq \frac 12$. 
\end{proof}
Now we are ready to prove the main proposition for the corrector step. 
\begin{proposition}  \label{prop:corr-ana}
Let $(x,\tau,y) \in Q_{DD}$ satisfy \eqref{eq:bound-omega}. Assume that the corrector step $d^c$ is calculated by  solving \eqref{eq:system-pred-1} with parameters defined in \eqref{eq:system-corrector-step-1}, and we choose $\alpha_1=\frac{\alpha_2}{\tau+\alpha_2 d_\tau}$ for the updates of \eqref{eq:update}. Consider $\bar \xi_1$ and $\bar \xi_2$ defined in \eqref{eq:corr-proof-1} and \eqref{eq:xi-1}, respectively. Then, for
\begin{eqnarray}  \label{eq:corrector-43}
\alpha_2:= \frac{1}{2(\bar \xi_4 + \bar \xi_2^2)},  \ \ \ \bar \xi_4:= 2\sqrt{\frac{1}{\xi-1}}  \left(1+\sqrt{\frac{1}{\xi-1}} \right) + \frac{\sqrt{\xi}+2}{\sqrt{\xi-1}},
\end{eqnarray}
we have
\begin{eqnarray}  \label{eq:corrector-44}
\Omega_\mu(x^+,\tau^+,y^+)-\Omega_\mu(x,\tau,y)  \leq - \frac{\alpha_2}{32 \bar \xi^2_1}. 
\end{eqnarray}
\end{proposition}
\begin{proof}
Assume that $d^c=[\bar d_x^\top \ d_\tau \ d_v^\top]$ is the corrector search direction. Then, by \eqref{eq:system-pred-1} and  \eqref{eq:system-corrector-step-1} we have 
\begin{eqnarray}  \label{eq:corrector-34}
(d^c)^\top U^\top  \cH U  d^c=\|U^\top \psi ^c+ \beta r^0\|^2_{(U^\top \cH U)^{-1}}.
\end{eqnarray}
Hence, we have inequality \eqref{eq:lem:ana-1-1} with $q:=\|U^\top \psi ^c+ \beta r^0\|^2_{(U^\top \cH U)^{-1}}$, and we already have the bounds \eqref{eq:dd-expand-3}, \eqref{eq:lem:corr-ana-6}, \eqref{eq:lem:corr-ana-7}, and \eqref{eq:corrector-37}. 
Here, we use \eqref{eq:lem:ana-1-1} to get another bound; if we consider the last term in the LHS of \eqref{eq:lem:ana-1-1}, we get
\begin{eqnarray}  \label{eq:corrector-38}
\frac{ \tau}{\mu}\left |  \langle d_y,Ax+\frac{1}{\tau} z^0 \rangle + \langle c, \bar d_x \rangle \right| \leq   (\sqrt{\xi}+1) \sqrt{\vartheta q}.
\end{eqnarray}
Note that from Corollary \ref{coro:bound-prox-1}, we have $\left\|\frac{\tau y}{\mu} - \Phi' \right\|_{[\Phi'']^{-1}} \leq \sigma(\Omega_\mu)$. Using this and \eqref{eq:corrector-38}, we have
\begin{eqnarray}  \label{eq:corrector-39}
\ \ \ \ \begin{array}{rcl}
&&\left|\frac{ d_\tau}{\mu}  \left ( \langle d_y,\Phi_*' \rangle + \langle c, \bar d_x \rangle \right) \right|  \\
&=&\left|\frac{ d_\tau}{\mu}  \left ( \langle d_y,Ax+\frac{1}{\tau} z^0 \rangle + \langle c, \bar d_x \rangle+\langle d_y,\Phi_*' - Ax+\frac{1}{\tau} z^0\rangle \right)\right|   \\
&\leq&\left|\frac{ d_\tau}{\tau}\right|  \left (\frac{ \tau}{\mu}\left |  \langle d_y,Ax+\frac{1}{\tau} z^0 \rangle + \langle c, \bar d_x \rangle \right| +  \frac{\tau}{\mu} \|d_y\|_{[\Phi'']^{-1}}\left\|\Phi_*' - Ax+\frac{1}{\tau} z^0\right\|_{\Phi''}\right)  \\
&\leq& \sqrt{\frac{q}{(\xi-1)\vartheta}} \left((\sqrt{\xi}+1) \sqrt{\vartheta q} +  \frac{\tau}{\mu} \|d_y\|_{[\Phi'']^{-1}}\left\|\Phi_*' - Ax+\frac{1}{\tau} z^0\right\|_{\Phi''} \right), \ \ \ \hfill \text{by \eqref{eq:dd-expand-2} and \eqref{eq:corrector-38},} \\
&\leq& \sqrt{\frac{q}{(\xi-1)\vartheta}} \left((\sqrt{\xi}+1) \sqrt{\vartheta q} +  \sqrt{q} \frac{\sigma(\Omega_\mu)}{1-\sigma(\Omega_\mu)} \right), \hfill \text{by \eqref{eq:dd-expand-5} and Lemma \ref{lem:der-to-real}},  \\
&\leq& \frac{\sqrt{\xi}+2}{\sqrt{\xi-1}} q, \hfill \text{for the case $\sigma(\Omega_\mu) \leq 0.5$}.
\end{array}
\end{eqnarray}

We want to work with the second inequality in \eqref{eq:corrector-16-2}. We already have a bound for $D(\alpha_2)$ in \eqref{eq:lem:corr-ana-D1} and we also have
$\rho(-t) \leq t^2$ for $t \in (0,0.6)$. 
By substituting  \eqref{eq:corrector-37} and \eqref{eq:corrector-39}, we get
\begin{eqnarray}  \label{eq:corrector-35}
\Omega_\mu(x^+,\tau^+,y^+)-\Omega_\mu(x,\tau,y)   \leq (-\alpha_2+(\bar \xi_4 + \bar \xi_2^2) \alpha_2^2)  \|U^\top \psi ^c+ \beta r^0\|^2_{(U^\top \cH U)^{-1}},
\end{eqnarray}
where $\bar \xi_4$ is defined in \eqref{eq:corrector-43}. 
If we choose $\alpha_2 \leq \frac{1}{2(\bar \xi_4 + \bar \xi_2^2)}$, then for the RHS we have
\begin{eqnarray}  \label{eq:corrector-42}
\leq -\frac{1}{2}  \alpha_2  \|U^\top \psi ^c+ \beta r^0\|^2_{(U^\top \cH U)^{-1}} \leq - \frac{\alpha_2}{32 \bar \xi^2_1}  \Omega_\mu,
\end{eqnarray}
where we used the bound for $ \|U^\top \psi ^c+ \beta r^0\|^2_{(U^\top \cH U)^{-1}}$ by Lemma \ref{lem:corr-ana}.
\end{proof}

\subsection{Complexity of following the path to $\mu=+\infty$}   \label{ch:com-result}
We have analyzed the predictor and corrector search directions in Section \ref{sec:analysis}. Now we can modify the statement of our predictor-corrector algorithm to one that provably follows the path in polynomial time. \\
\noindent\rule{16cm}{0.6pt}\\
\noindent \textbf{Polynomial-time Predictor-Corrector Algorithm  (PtPCA)} \vspace{-0.3cm} \\
\noindent\rule{16cm}{0.6pt}\\
\noindent {\bf Initialization:} Choose $z^0 \in \inte D$ and set $y^0:=\Phi'(z^0)$. Set $x^0:=0$,  $\tau_0:=1$, $\mu_0:=\mu(x^0,\tau_0,y^0)$, and $k=0$. Choose a constant $\xi > 1$  and constants $0 < 4 \delta_1 < \delta_2  \leq \frac{1}{100 \left((\bar \xi_2 \bar \xi_1 )^3+\bar \xi_3 \bar \xi_1^3\right)^2 }$, where $\bar \xi_1$, $\bar \xi_2$, and $\bar \xi_3$ are functions of $\xi$ defined in \eqref{eq:corr-proof-1} and \eqref{eq:xi-1}.   \\
\noindent {\bf while} (the stopping criteria are not met) 
\begin{addmargin}{1cm}
 {\bf if} ($\Omega_{\mu_k}(x^k,\tau_k,y^k) > \delta_1$)
\end{addmargin} 
\begin{addmargin}{1.5cm}
Calculate the corrector search direction $(d_x,d_\tau,d_y)$ by \eqref{eq:system-pred-1} with $r_{RHS}$ and $\hat H$ defined in \eqref{eq:system-corrector-step-1}, and choose $\alpha_2$ as in \eqref{eq:corrector-43} and $\alpha_1:=\frac{\alpha_2}{\tau+\alpha_2 d_\tau}$.  Apply the update in \eqref{eq:update} to get  $(x^{k+1},\tau_{k+1},y^{k+1})$, and  define $\mu_{k+1}:=\mu_k$.
\end{addmargin} 
\begin{addmargin}{1cm}
{\bf if} ($\Omega_{\mu_k}(x^k,\tau_k,y^k) \leq \delta_1$)
\end{addmargin}  
\begin{addmargin}{1.5cm}
Calculate the predictor search direction $(d_x,d_\tau,d_y)$ by \eqref{eq:system-pred-1} with $r_{RHS}=r^0/\mu_k^2$ and any $\hat H$ that satisfies \eqref{eq:lem:hessian-close-1-2}, and choose $\alpha_2=\frac{\kappa_1}{\sqrt{\vartheta}} \mu_k$ for $\kappa_1$ defined in the proof of Proposition \ref{prop:dd-4}, and $\alpha_1:=\frac{\alpha_2}{\tau+\alpha_2 d_\tau}$.  Apply the update in \eqref{eq:update} to get  $(x^{k+1},\tau_{k+1},y^{k+1})$, and  define $\mu_{k+1}:=\mu(x^{k+1},\tau_{k+1},y^{k+1})$. 
\end{addmargin}
\begin{addmargin}{1cm}
$k \leftarrow k+1$.
\end{addmargin}
\noindent {\bf end while}\\
\noindent\rule{16cm}{0.6pt}\\
Note that even though the choices of $\delta_1$ and $\delta_2$ in the PtPCA, as we show in the following, gives us the desired iteration complexity bounds, these choices are too small for practical purposes. In practice, as we have done in the DDS code, $\delta_1$ and $\delta_2$  are chosen large enough to guarantee long steps. To achieve long steps in practice, we should not restrict the algorithm to Dikin ellipsoids. There are properties for classes of s.c.\ barriers that strengthen the Dikin ellipsoid property to anywhere in the interior of the domain. We mention \emph{negative curvature} \cite{nesterov1997self,guler1997hyperbolic,nesterov2016local} and \emph{$\alpha$-regularity} \cite{multi} here. Negative curvature is a property for many interesting LH s.c.\ barriers (see \cite{nesterov1997self}, \cite{guler1997hyperbolic}, and \cite{nesterov2016local}-Section 9.2) that lets us extend a Hessian estimation property like  \eqref{property-3} to effectively the whole domain of the s.c.\ barrier. A s.c.\ function is additionally $\alpha$-regular if the second derivative also controls the fourth derivative in a proper way \cite{multi}. It was shown in \cite{multi} that many useful s.c.\ barriers are $\alpha$-regular, such as the ones in Table \ref{tbl:DD-example-1} for LP, SOCP, and SDP, and the ones we built for Geometric Programing and Entropy Programming. If all the s.c.\ barriers given in a problem instance have one of these properties, the practical version of our algorithm is theoretically guaranteed to take long steps (a large portion of the distance between the current iterate and the boundary). If even one of these barriers does not have any long-step property, this theoretical guarantee may not hold. It is possible to construct some pathological examples on which the algorithm has to take a short step in every iteration; however, the practical version of the algorithm generally has a chance to take long steps in most of the iterations.

Our analysis of the predictor and corrector steps implies the following theorem:
\begin{theorem} \label{thm:complexity-result}
For the polynomial-time predictor-corrector algorithm, there exists a positive constant $\kappa_2$ depending on $\xi$  such that after $N$ iterations, we get a point $(x,\tau,y) \in Q_{DD}$ such that
\begin{eqnarray} \label{eq:thm-com}
\mu(x,\tau,y) \geq \exp\left(\frac{\kappa_2}{\sqrt{\vartheta}} N \right). 
\end{eqnarray}
\end{theorem}
\begin{proof}
By Proposition \ref{prop:corr-ana}, after each predictor step, we have to do at most
\[
64(\bar \xi_4 + \bar \xi_2^2)\bar \xi_1^2  (\delta_2-\delta_1),
\]
number of corrector steps to satisfy $\Omega_\mu(x,\tau,y) \leq \delta_1$. Also, by Propositions \ref{prop:mu-increase} and  \ref{prop:dd-4}, after $\bar N$ cycles of predictor-corrector steps, we have
\[
\mu  \geq \left( 1+\frac{\kappa_1}{\sqrt{ \vartheta}} \right)^{\bar N}. 
\]
Therefore, we have \eqref{eq:thm-com} for $\kappa_2 = O(1) \kappa_1$. 
\end{proof}

Theorem \ref{thm:complexity-result} is the core of several consequences about determining the statuses of the problem in polynomial time (see \cite{karimi_status_arxiv}). In this article, we briefly discuss the case where the problem and its dual both are strictly feasible. In this case, we can define a \emph{feasibility measure} $\sigma_f$ (which is a complexity measure) that represents how good the geometry of the feasible regions are and the proximity of $z^0$ and $y^0$ to the boundaries of their respective domains, and prove the following theorem about the connection between $\tau$ and $\mu$: 
\begin{theorem} [\cite{karimi_status_arxiv, karimi_thesis}] \label{thm:fes-meas-dd}
Assume that both primal and dual are strictly feasible and for a point $(x,\tau,y) \in Q_{DD}$  we have the additional property that $\delta_*(y|D)+y_{\tau,0}+\tau \langle c,x \rangle \leq 0$. Then, 
\begin{eqnarray} \label{eq:thm:fes-meas-dd-1}
\tau-1  \geq  \sigma_f \mu(x,\tau,y) - \frac{1}{\sigma_f},
\end{eqnarray}
where $\sigma_f$ is the feasibility measure defined as 
\begin{eqnarray*} \label{eq:dd-outcome-30}
\sigma_f := \sup \left\{\alpha: \alpha <  1, \ \ \bar y - \alpha y^0  \in D_*,  \ \ \frac{A \bar x - \alpha z^0}{ 1- \alpha}  \in D, \ \ \delta_* (\bar y - \alpha y^0|D) +  \bar y_\tau - \alpha y _{\tau,0}  \leq 0   \right\},
\end{eqnarray*}
for $\bar x := \argmin_x  \{\Phi(Ax)+\langle c,x \rangle \}$, $\bar y := \Phi'(A \bar x)$, and $\bar y_\tau := -\xi \vartheta - \langle \bar y, A \bar x \rangle$. 
\end{theorem}
Note that by Lemma \ref{lem:dg-bound-1}, the hypothesis of the above theorem holds for the points close to the central path. 
Putting together the discussion we had in Subsection \ref{subsec:inter} and Theorem \ref{thm:complexity-result}, we conclude that when we have strict primal and dual feasibility, 
in $O\left( \sqrt{\vartheta}\ln\left(\frac{\vartheta}{\epsilon}\right)\right)$ number of iterations, we obtain an $\epsilon$-solution of the problem.

\section{Conclusions} \label{sec:con}

After introducing the Domain-Driven setup, we defined an infeasible-start primal-dual central path and designed and analyzed algorithms that can follow this path efficiently (Theorem \ref{thm:complexity-result}). Following our discussion in Subsection \ref{subsec:inter}, the important question is: for different statuses of the problem, what is the behavior of $(x,\tau,y)$ when $\mu \rightarrow +\infty$, and for which values of $\mu$ we can determine the status of the problem with $\epsilon$ accuracy using $(x,\tau,y)$? We answered this question for the case of strict primal and dual feasibility, for which our algorithm can return an $\epsilon$-solution in $O\left( \sqrt{\vartheta}\ln\left(\frac{\vartheta}{\epsilon}\right)\right)$ number of iterations. This bound is the current best and is new for the type of formulations we used for  handling infeasibility, even in the special case of SDP. 

The geometry of a problem in the Domain-Driven form and possible different statuses are discussed in \cite{karimi_status_arxiv} and it is shown  that the PtPCA algorithm returns certificates (heavily relying on duality) for each of these statuses in polynomial time. The iteration complexity bounds are comparable to the current best ones we have for the conic formulations (to the best of our knowledge mostly in  \cite{infea-2}). The algorithms of this article are the base of a code, called DDS (Domain-Driven Solver), that solves many classes of problems, including those listed in Section \ref{introduction}, and the list is expanding. 

An interesting special case of the Domain-Driven formulation is when $D=K-b$ where $K$ is a convex cone equipped with a $\vartheta$-LH.s.c.\ barrier $\hat \Phi$  and $b \in \R^n$. Then, the recession cone of $D$ is $K$ and so $D_*$ is the dual cone of $K$, called $K_*$. We have $\Phi_*(y)=-\langle b,y\rangle+\hat \Phi_*(y)$, where $\hat \Phi_*$ is the LF conjugate of $\hat \Phi$ and is also a LH-s.c.\ barrier. We get many simplifications by using the properties of the cones and LH-s.c.\ barriers. For example, $\delta_*(y|D)=-\langle b,y \rangle$ and so the duality gap reduces to the classic conic duality gap $\langle c, x\rangle-\langle b,y \rangle$. Another simplification  is that inequality \eqref{eq:norm-phi*-1} that we use frequently in our analysis becomes equality as
\[
\langle y, \Phi_*''(y) y \rangle = \langle y, \hat \Phi_*''(y) y \rangle  = \vartheta. 
\]
These simplifications stand out in the status determination analyses \cite{karimi_status_arxiv} and we show that, in this case, our complexity results are at least as good as the ones in \cite{infea-2} and recover them. 

\appendix 
\section{Self-Concordant Functions}  \label{appen-s.c.}
The reader can refer to \cite{interior-book},  \cite{nemirovski-notes}, and \cite{lectures-book} for a comprehensive study of the properties and calculus of s.c.\ functions, or to \cite{cone-free} and \cite{karimi_thesis}-Chapter 4 for a summery of more important properties. In this section, we summarize the properties of self-concordant (s.c.) functions that we use in this paper. 

\subsection{Self-concordant (s.c.) functions} \label{s.c.f. properties}
A convex function $f : \mathbb E \rightarrow \ \mathbb R \cup \{+\infty\}$ is called $a$-s.c.\ function if its domain $Q$ is open, $f$ is $\mathcal C^3$ on $Q$ and 
\begin{enumerate} [(i)]
	\item $f(x_i) \rightarrow +\infty$ for every sequence $\{x_i\} \subset Q$ that converges to a point on the boundary of $Q$.
	\item There exists a positive real constant $a$ such that 
	\begin{eqnarray} \label{property-1}
	|f'''(x) [h,h,h]| \leq 2 a^{-1/2}(f''(x) [h,h])^{3/2} = 2 a^{-1/2} \|h\|^{3}_{f''(x)}, \ \ \forall(x \in Q, h\in \mathbb E),
	\end{eqnarray}
where $ f^k(x) [h_1,\ldots,h_k]$ henceforth is the value of the $k$th differential of $f$ along directions $h_1,\ldots,h_k \in \mathbb E$. 
\end{enumerate}
We say that $f$ is \emph{non-degenerate} if its Hessian $f''(x)$ is positive definite at some point (and then it can be proved to be positive definite at all points) in $Q$.
From now on, we assume that $f$ is a s.c.\ function with domain $Q$.

\noindent\textbf{(Behaviour in Dikin ellipsoid and some basic inequalities):}
\begin{enumerate} [(a)]
	\item For every point $x \in Q$, we define the \emph{ Dikin ellipsoid} centered at $x$ as
	\begin{eqnarray*}
		W_1(x) := \left\{y \in \mathbb E : \frac{1}{\sqrt{a}} \|y-x\|_{f''(x)} \leq 1\right\}.
	\end{eqnarray*} 
	Then we have $W_1(x) \subset Q$ and for every point $y \in W_1(x)$ we can estimate the Hessian of $f$ at $y$ in term of the Hessian of $f$ at $x$ as
	\begin{eqnarray} \label{property-3}
	(1-r)^2 f''(x) \preceq f''(y) \preceq \frac{1}{(1-r)^2} f''(x),
	\end{eqnarray}
	where $r:= \frac{1}{\sqrt{a}} \|y-x\|_{f''(x)}$. For a proof see \cite{interior-book}-Theorem 2.1.1.
	\item For every point $x,y \in Q$ and for $r:=\frac{1}{\sqrt{a}}\|y-x\|_{f''(x)}$, we have
	\begin{eqnarray} \label{property-4}
	f(y) &\geq& f(x) + \langle f'(x),y-x \rangle + a\rho(r),  \nonumber \\
	f(y) &\leq& f(x) + \langle f'(x),y-x \rangle + a\rho(-r),
	\end{eqnarray}
where $\rho(\cdot)$ is defined in \eqref{eq:rho}. For the proof see \cite{lectures-book}-Chapter 5. 
\end{enumerate}

\noindent \textbf{(LF conjugate of a s.c.\ function):}
Let $f: \mathbb E \rightarrow \mathbb R \cup \{+\infty\}$ be convex. The \emph{Legendre-Fenchel (LF) conjugate} of $f$ is defined as
\begin{eqnarray}  \label{eq:LF-1}
f_*(y) := \sup_{x} \{\langle y,x \rangle-f(x)\}. 
\end{eqnarray}  
$f_*$ is always a convex function and its domain is all the points that \eqref{eq:LF-1} has a bounded solution. For a proper convex function, we have $(f_*)_*=f$ if and only if the epigraph of $f$ is closed ($f$ is a closed convex function), see for example \cite{convex-analysis}.  We use the following well-known fact frequently in this paper.
\begin{theorem}\label{thm:FY} \emph{\bf (Fenchel-Young inequality)} Let $f: \mathbb E \rightarrow \mathbb R \cup \{+\infty\}$ be a convex function and $f_*$ be its LF conjugate. For every point $x$ in the domain of $f$ and every $y$ in the domain of $f_*$, we have
\begin{eqnarray} 
f(x) + f_*(y)  \geq \langle y,x \rangle. 
\end{eqnarray}
Equality holds if and only if $y \in \partial f(x)$.
\end{theorem}
Assume that $f(x)$ is differentiable and the optimal value of \eqref{eq:LF-1} for $\bar y$ is attained at $\bar x$, then we must have $\bar y =  f'(\bar x)$. By Theorem \ref{thm:FY}, if both $f$ and $f_*$ are twice differentiable, for every point $x$ in the domain of $f$ we have
\begin{eqnarray} \label{eq:LF-2}
 x= f'_*(f'(x))  \ \   \Rightarrow \ \ \  f_*''( f'(x)) = [f''(x)]^{-1}. 
 \end{eqnarray}

Let $Q_*$ be the domain of $f_*$; the set of all points for which the right hand side of \eqref{eq:LF-1} is finite. We mentioned that $Q_*$ is convex and $f_*$ is a convex function on $Q_*$. It is shown in \cite{interior-book}- Section 2.4 that $Q_*=f'(Q)$, $f_*$ is a non-degenerate s.c.\ function and the LF conjugate of $f_*$ is exactly $f$.

\begin{lemma}   \label{lem:der-to-real}
Let $f$ be a 1-s.c.\ function. For every $x$ and $y$ in the domain of $f$ which satisfy  $r:= \|x-y\|_{f''(x)} <1$ we have
\begin{eqnarray}  \label{eq:lem:der-to-real-1}
 \|f'(x)-f'(y)\|^*_{f''(x)}  \leq  \frac{r}{1-r}.
\end{eqnarray}
\end{lemma}
\begin{proof}
Let us define $q:=y-x$. Starting with the fundamental theorem of calculus, we have:
\begin{eqnarray*}
\begin{array}{rcl}
\|f'(x)-f'(y)\|^*_{f''(x)}  &=& \left\|\int_0^1 f''(x+tq) q dt \right\|^*_{f''(x)} 
\leq  \int_0^1 \| f''(x+tq) q \|^*_{f''(x)} dt \nonumber \\
                 & \underbrace{\leq}_{\text{\eqref{property-3}}} & \int_0^1 \frac{1}{1- \|tq\|_{f''(x)}}\| f''(x+tq) q \|^*_{f''(x+tq)} dt \nonumber \\
                 & = & \int_0^1 \frac{1}{1- \|tq\|_{f''(x)}}\| q \|_{f''(x+tq)} dt  
                  \underbrace{\leq}_{\text{\eqref{property-3}}}  \left ( \int_0^1 \frac{1}{(1- tr)^2} dt \right ) r = \frac{r}{1-r}.
 \end{array}
\end{eqnarray*}
\end{proof}

\subsection{Self-concordant (s.c.) barriers}   \label{s.c.b. properties}
For a $\vartheta \geq 1$, we say that a $1$-s.c.\ function is a $\vartheta$-s.c.\ barrier for $\cl(Q)$ if we have
\begin{eqnarray} \label{property-2}
| f'(x) [h]| \leq \sqrt{\vartheta} \|h\|_{f''(x)}, \ \ \forall(x \in Q, h\in \mathbb E).
\end{eqnarray}
A non-degenerate s.c.\ function $f$ is a $\vartheta$-s.c.\ barrier if and only if 
\begin{eqnarray} \label{property-2-2}
\|f'(x)\|_{[f''(x)]^{-1}} \leq \sqrt{\vartheta}, \  \ \ \forall  x \in Q. 
\end{eqnarray}
If $Q$ is a convex cone, we say $f$ is $\vartheta$-logarithmically-homogeneous if for every $x \in Q$, we have
\begin{eqnarray} \label{eq:LH}
f(tx)=f(x)-\vartheta \ln(t), \ \ \forall ( t>0).
\end{eqnarray}
\noindent \textbf{(Basic properties of s.c.\ barrier's):} Let $f$ be a $\vartheta$-s.c.\ barrier, then the following inequalities hold for every pair $x,y \in Q$ (see \cite{interior-book}-Proposition 2.3.2 and  \cite{nemirovski-notes}-Chapter 3):
\begin{eqnarray} \label{property-6}
f'(x)[y-x] \leq \vartheta;
\end{eqnarray}
where, as before, $f'(x)[h]$ is the first order differential of $f$ taken at $x$ along the direction $h$. 
$f$ is non-degenerate if and only if $Q$ does not contain lines. $f$ is bounded below if and only if $Q$ is bounded. Then, $f$ is non-degenerate and attains its unique minimizer $x_f$ on  $Q$.
\begin{lemma}  \label{lem:dd-9}
	Let $\Phi$ be a $\vartheta$-s.c.\ barrier with domain $\inte D \subset \mathbb E$, and $\xi >1$. Then, the function $ \Phi \left ( \frac{z}{\tau} \right)-\xi \vartheta \ln(\tau)$ with domain $\{(z,\tau): \tau > 0, \frac{z}{\tau} \in \inte D\}$ is a $\bar \xi$-s.c.\ function for an absolute constant $\bar \xi$ depending on $\xi$. Moreover, its LF conjugate  and also the summation of $ \Phi \left ( \frac{z}{\tau} \right)-\xi \vartheta \ln(\tau)$ with its LF conjugate are also $\bar \xi$-s.c.\ functions.
\end{lemma}
\begin{proof}
Consider the function $\Phi(\frac{z}{\tau})-\xi \vartheta \ln(\tau)$. First we show that the function is convex. Let us define
\[
g(\alpha):= \Phi \left(\frac{z+\alpha d_z}{\tau+\alpha d_\tau} \right)-\xi \vartheta \ln(\tau + \alpha d_\tau). 
\] 
Then, we have
\begin{eqnarray*} \label{eq:dd-35}
g''(0)&=&\frac{1}{\tau^2} \left [ \langle d_z-\frac{d_\tau}{\tau} z, \Phi'' \left (\frac{z}{\tau} \right) \left ( d_z-\frac{d_\tau}{\tau} z \right) \rangle   - 2d_\tau \langle \Phi' \left (\frac{z}{\tau} \right),   d_z-\frac{d_\tau}{\tau} z \rangle + \xi \vartheta d_\tau^2 \right].  
\end{eqnarray*}
By using inequality \eqref{property-2} for the middle term and doing some simple algebra we get
\begin{eqnarray} \label{eq:dd-35-2}
g''(0) \geq \frac{1}{\tau^2} \left [ \left \|d_z-\frac{d_\tau}{\tau} z \right \|_{\Phi''} - |d_\tau| \sqrt{\vartheta}  \right]^2+ (\xi-1) \frac{d_\tau^2}{\tau^2} \vartheta .
\end{eqnarray}
\eqref{eq:dd-35-2} shows that  $\Phi(\frac{z}{\tau})-\xi \vartheta \ln(\tau)$ is strictly convex for every $\xi > 1$. 

To prove that it is a s.c.\ function, we show that there exists an absolute constant $\bar \xi$ depending on $\xi$ such that  $|g'''(0)| \leq  2 \bar \xi^{-1/2} (g''(0))^{3/2}$. 
For simplicity, let us define $h:=\frac{1}{\tau} \left( d_z-\frac{d_\tau}{\tau} z \right)$. First, note that from \eqref{eq:dd-35-2} we have
\begin{eqnarray} \label{eq:lem:dd-9-1}
\left |\frac{d_\tau}{\tau} \sqrt{\vartheta} \right |  \leq \frac{\sqrt{g''(0)}}{\sqrt{\xi-1}},  \ \text{and} \ 
\|h\|_{\Phi''}   \leq \sqrt{g''(0)} + \left |\frac{d_\tau}{\tau} \sqrt{\vartheta} \right |  \leq  \underbrace{\left ( 1+\frac{1}{\sqrt{\xi-1}} \right)}_{=:\gamma}\sqrt{g''(0)}.
\end{eqnarray}
By expanding the expression for $g'''(0)$, we have
\begin{eqnarray} \label{eq:lem:dd-9-2}
g'''(0)=\Phi'''[h,h,h]-6\Phi''[h,h] \left( \frac{d_\tau}{\tau}\right)+6\Phi'[h] \left( \frac{d_\tau}{\tau}\right)^2-2 \xi \vartheta \left( \frac{d_\tau}{\tau}\right)^3.
\end{eqnarray}
Because $\Phi$ is a 1-s.c.\ function, by definition in \eqref{property-1}, we have $|\Phi'''[h,h,h] | \leq 2 (\Phi''[h,h])^{3/2}=2(\|h\|_{\Phi''})^3$, and because $\Phi$ is a $\vartheta$-s.c barrier, by definition \eqref{property-2}, we have $|\Phi'[h]| \leq \sqrt{\vartheta} \|h\|_{\Phi''}$. Substituting these in \eqref{eq:lem:dd-9-2}, using the inequalities in \eqref{eq:lem:dd-9-1} and the fact that $\vartheta \geq 1$, we have:
\begin{eqnarray} \label{eq:lem:dd-9-3}
g'''(0)\leq \left ( 2\gamma^3+\frac{6\gamma^2}{\sqrt{\xi-1}}+\frac{6\gamma}{\xi-1}+\frac{2\xi}{(\xi-1)^{3/2}} \right)(g''(0))^{3/2},
\end{eqnarray}
where $\gamma$ is defined in \eqref{eq:lem:dd-9-1}. 
	
	For the second part of the lemma for the conjugate function, see the proof of Theorem 2.4.1 in \cite{interior-book}.
\end{proof}
\subsection{LF conjugate of s.c.\ barriers}  \label{appen-s.c.-LF}
If $f$ is a $\vartheta$-s.c.\ barrier, then $f_*$ is a s.c.\ function, but it is \textbf{not} necessarily a s.c.\ barrier.  $Q_*$ is either the entire $\mathbb E^*$ if $Q$ is bounded, or the open cone
\begin{eqnarray} \label{open-cone}
\rec_*(Q):= \{s \in \mathbb E^* : \langle s,h \rangle  < 0, \forall h \in \rec(Q) \},
\end{eqnarray}
where $\rec(Q)$ is the \emph{recession cone} of $Q$ defined as
\begin{eqnarray} \label{eq:rec-cone}
\rec(Q):= \{h \in \mathbb E : x+th \in Q, \ \ \forall x \in Q, \  \forall t \geq 0 \}.
\end{eqnarray}
In this article, we frequently use the fact that $f_*$ has some useful properties beyond those of an arbitrary s.c.\ function, such as Theorem \ref{thm:support-fun}.


\section{Examples of s.c. functions to clarify Figure \ref{Fig-diagram}} \label{appen:examples}
It is well-known that $-\ln(x)$ is a 1-LH s.c.\ barrier for the cone $\R_{+}$ and its LF conjugate $-1-\ln(-y)$ is also a 1-LH s.c.\ barrier.  Assume that $f: \R^n \rightarrow \R$ is a convex function with the LF conjugate $f_*$. Then, we can easily verify that for every $b \in \R^n$, the LF conjugate of $f(x-b)$ is $\langle b,y \rangle + f_*(y)$. Consider the following univariate function and its LF conjugate:
\begin{eqnarray*}
f(x):=-\ln(x-1), \ \ \ f_*(y)=-1+y-\ln(-y). 
\end{eqnarray*}
$f(x)$ is a 1-s.c.\ barrier. $f_*(y)$ is a s.c. function, but is not a s.c.\ barrier. 

As it is shown in Figure \ref{Fig-diagram}, if a function is LH s.c.\ barrier, its LF conjugate is also a LH s.c.\ barrier \cite{interior-book}. A question is: does there exist a s.c.\ barrier $f$ that is not LH, while its LF conjugate $f_*$ is also a s.c.\ barrier, as implied in  Figure \ref{Fig-diagram}? Note that by Subsection \ref{appen-s.c.-LF}, the domains of $f$ and $f_*$ both must be convex cones. The following theorem shows that the answer is yes:
\begin{theorem}
Let $n$ be a positive integer. Assume that $f(x)$ is a non-degenerate  $\vartheta$-LH s.c.\ barrier with domain $K \subset \R^n$ and let $A: \R^n \rightarrow \R^n$ be a linear transformation such that $\{0\} \subsetneq AK \subseteq K$. Then, for every $b\in \inte K$, the function $g(x):=f(Ax+b)+f(x)$ is a $2\vartheta$-s.c.\ barrier, it is not logarithmically homogeneous, and its LF conjugate $g_*$ is also a s.c.\ barrier. 
\end{theorem}
\begin{proof}
We know that $g$ is a $2\vartheta$-s.c.\ barrier \cite{interior-book} with domain $K$, and $g$ is not logarithmically homogeneous, since otherwise we must have $f(tAx+b)=f(Ax+b)-k\ln(t)$ for a fixed $k>0$ and every $t>0$,  which gets violated when $t$ tends to zero. To show $g_*$ is also a s.c.\ barrier, we need to prove that $\langle g_*'(y), [g_*''(y)]^{-1}g_*'(y)\rangle$  is bounded by an absolute constant for every $y \in K_*$. For a given $y$, let $x:=g'_*(y)$, then by the properties of LF conjugate, we have
\begin{eqnarray} \label{eq:Appen-B-1}
\langle g_*'(y), [g_*''(y)]^{-1}g_*'(y)\rangle&=&\langle x,g''(x)x\rangle \nonumber \\
&=& \langle Ax,f''(Ax+b)Ax\rangle+\langle x,f''(x)x\rangle  \nonumber \\
&\leq& (\vartheta+2\sqrt{\vartheta})^2 \langle Ax,f''(Ax)Ax\rangle+\langle x,f''(x)x\rangle \nonumber \\
&=& (\vartheta+2\sqrt{\vartheta})^2 \vartheta +\vartheta,  \ \ \ \text{\cite{interior-book}-eq (2.3.14)}.
\end{eqnarray}
For the inequality above, we used equation \cite{nemirovski-notes}-(3.16) and also the fact that $Ax+\alpha b \in K$ for all $\alpha \in \R_+$ and so $\pi_{Ax}(Ax+b)=0$, where $\pi$ is the Minkowski function of $K$ (defined in \cite{interior-book}-Subsection 2.3.2 or \cite{nemirovski-notes}). Inequality \eqref{eq:Appen-B-1} confirms that $g_*$ is a $((\vartheta+2\sqrt{\vartheta})^2 \vartheta +\vartheta)$-s.c.\ barrier. 
\end{proof}
As an example, consider $f(x):=-\sum_{i=1}^m \ln(a_i^\top x)$ for $a_i \in \R^n$, $i \in \{1,\ldots,m\}$, which is a $m$-LH s.c.\ barrier. Then, the function $g(x):=-\sum_{i=1}^m \ln(a_i^\top x)-\sum_{i=1}^m \ln(a_i^\top x+1)$ is a $2m$-s.c.\ barrier that is not LH and $g_*$ is also a s.c.\ barrier.

\renewcommand{\baselinestretch}{1}
\bibliographystyle{siam}
\bibliography{References}

\end{document}